\documentclass[12pt,a4]{amsart}
\usepackage{amsmath}
\usepackage{amssymb}
\usepackage{amsthm}
\usepackage{enumitem}
\usepackage{url}
\usepackage{soul}
\usepackage{times}
\usepackage[T1]{fontenc}
\usepackage[export]{adjustbox}

\usepackage{color}
\usepackage{dsfont}
\usepackage{epsfig}
\usepackage[hidelinks]{hyperref}
\usepackage{setspace}
\usepackage{epstopdf}
\usepackage[authoryear,round,sort]{natbib}
\usepackage{comment}
\usepackage{booktabs}
\usepackage{bm}
\usepackage{float}
\usepackage[linesnumbered,ruled,vlined]{algorithm2e}
\SetKwInput{KwInput}{Input}                
\SetKwInput{KwOutput}{Output} 
\usepackage{stmaryrd}
\usepackage{subcaption,graphicx}
\usepackage{float}
\numberwithin{equation}{section}

\usepackage[font=small,labelfont=bf]{caption}
\DeclareCaptionFont{tiny}{\tiny}
\usepackage{geometry}
\usepackage{graphicx}
\theoremstyle{plain}

\newtheorem{theorem}{Theorem}[section]

\newtheorem{corollary}{Corollary}[section]

\theoremstyle{definition}
\newtheorem{definition}{Definition}[section]

\theoremstyle{remark}
\newtheorem{rk}{Remark}[section]
\expandafter\let\expandafter\oldproof\csname\string\proof\endcsname
\let\oldendproof\endproof
\renewenvironment{proof}[1][\proofname]{%
  \oldproof[\noindent\textbf{#1.} ]%
}{\oldendproof}

\newcommand{\be}{\begin{equation}}
\newcommand{\ee}{\end{equation}}
\newcommand{\by}{\begin{eqnarray*}}
\newcommand{\ey}{\end{eqnarray*}}

\DeclareMathOperator*{\argmin}{arg\,min}

\renewcommand{\leq}{\leqslant}
\renewcommand{\geq}{\geqslant}
\usepackage{xcolor}
\definecolor{dark-red}{rgb}{0.4,0.15,0.15}
\definecolor{dark-blue}{rgb}{0.15,0.15,0.4}
\definecolor{medium-blue}{rgb}{0,0,0.5}
\hypersetup{
    colorlinks, linkcolor={red},
    citecolor={blue}, urlcolor={blue}
}
\allowdisplaybreaks
\title[New reversibilizations via geometric projections, generalized mean and convex $*$-conjugate]{Systematic approaches to generate reversiblizations of Markov chains}
\author{Michael C.H. Choi}
\address{Department of Statistics and Data Science and Yale-NUS College, National University of Singapore, Singapore}
\email{mchchoi@nus.edu.sg}

\author{Geoffrey Wolfer}
\address{RIKEN Center for AI Project, Tokyo, Japan}
\email{geoffrey.wolfer@riken.jp}

\begin{document}

\date{\today}
\maketitle

\begin{abstract}
	Given a target distribution $\pi$ and an arbitrary Markov infinitesimal generator $L$ on a finite state space $\mathcal{X}$, we develop three structured and inter-related approaches to generate new reversiblizations from $L$. The first approach hinges on a geometric perspective, in which we view reversiblizations as projections onto the space of $\pi$-reversible generators under suitable information divergences such as $f$-divergences. With different choices of functions $f$, we not only recover nearly all established reversiblizations but also unravel and generate new reversiblizations. Along the way, we unveil interesting geometric results such as bisection properties, Pythagorean identities, parallelogram laws and a Markov chain counterpart of the arithmetic-geometric-harmonic mean inequality governing these reversiblizations. This further serves as motivation for introducing the notion of information centroids of a sequence of Markov chains and to give conditions for their existence and uniqueness. Building upon the first approach, we view reversiblizations as generalized means. In this second approach, we construct new reversiblizations via different natural notions of generalized means
	such as the Cauchy mean or the dual mean. In the third approach, we combine the recently introduced locally-balanced Markov processes framework and the notion of convex $*$-conjugate in the study of $f$-divergence. The latter offers a rich source of balancing functions to generate new reversiblizations.
	\smallskip
	
	
\end{abstract}

\tableofcontents


\section{Introduction}\label{sec:intro}

Given a target distribution $\pi$, and an arbitrary Markov infinitesimal generator $L$ on a finite state space $\mathcal{X}$, what are the different ways to reversiblize $L$, i.e. to transform $L$ so that it becomes $\pi$-reversible? In the classical text,  \cite{AF14} introduce three types of reversiblizations, namely the additive reversiblization, the multiplicative reversiblization in the discrete-time setting (i.e. a transition matrix $P$ multiplied by its $\pi$-dual, say $P^*$), and the Metropolis-Hastings reversiblization. \textcolor{black}{Higher order multiplicative reversiblizations have also been investigated in the literature, for instance in \cite{M97} for long-time convergence of simulated annealing, in \cite{Paulin15} for pseudo-spectral gap and concentration inequalities of non-reversible Markov chains, and the second Metropolis-Hastings reversiblization is proposed and investigated in \cite{Choi16,CH18}}. The geometric mean reversiblization is analyzed in \cite{DM09} in continuous-time and in \cite{WW21} in discrete-time, while the Barker proposal (which is in fact a harmonic mean reversiblization, see Section \ref{subsec:alphadiv} below) has recently enjoyed considerable interest in the Markov chain Monte Carlo literature in a series of papers \citep{Z20,LZ22,VLZ22}.

The study of reversiblizations is an important subject from at least the following three perspectives. First, reversiblizations, as a tool, allows us to study various properties of non-reversible Markov chains by analyzing their reversible counterpart. Owing to the absence of symmetry in the original non-reversible chain, we take advantage of the symmetrical properties of its reversiblized counterpart to effectively understand features of the original non-reversible chain such as its rate of convergence to equilibrium. This follows from the seminal paper by \cite{Fill91} and subsequent papers such as \cite{Paulin15,Choi16}, in which this spirit of reversiblizations is used to define the pseudo-spectral gap to analyze long-time convergence to $\pi$ under distances such as the total variation distance. Second, the study of reversiblizations may yield improved stochastic algorithms for sampling from $\pi$, a setting that commonly arises in applications such as Bayesian statistics. In the classical Metropolis-Hastings algorithm, one takes in a target distribution $\pi$ and a proposal chain with generator $L$ to transform $L$ into a generator that is reversible with respect to the target $\pi$. Thus, in a broad sense, sampling from $\pi$ amounts to reversiblizing a given proposal generator $L$, and being able to generate new reversiblizations may inspire design of improved stochastic algorithms, see for instance the second Metropolis-Hastings generator in \cite{CH18} or the Barker proposal in \cite{Z20,LZ22,VLZ22}. Third, new reversiblizations also give rise to new symmetrizations of non-symmetric and non-negative matrices or in general non-self-adjoint kernel operators. By taking $\pi$ to be the discrete uniform distribution on $\mathcal{X}$, this yields symmetrizations of the original non-symmetric and non-negative matrices. To the best of our knowledge, many of the new reversiblizations or symmetrizations proposed in subsequent sections of this manuscript have not yet been investigated in the linear algebra or functional analysis literature.

Despite the existence of numerous reversiblizations in literature, there is a lack of systematic approaches for generating new ones. In this paper, we introduce three structured methodologies that not only generate new reversiblizations but also recover most of the established ones. We summarize our main contributions as follow:

\begin{enumerate}
	\item \textbf{Generating reversiblizations via geometric projections.} This approach continues the line of work initiated in \cite{BD01,DM09,WW21}, in which reversiblizations are viewed as projections under information divergences such as $f$-divergences. The advantage of this approach is that we can recover all known reversiblizations in a unified framework. We also discover that the Barker proposal arises naturally as a projection under the $\chi^2$-divergence. Notable highlights of this approach include bisection properties, Pythagorean identities, parallelogram laws and a Markov chain counterpart of the arithmetic-geometric-harmonic mean (AM-GM-HM) inequality for various hitting time and mixing time parameters. We also introduce, visualize and characterize the notion of $f$ and $f^*$-projection centroids of a sequence of Markov chains.
	
	\item \textbf{Generating reversiblizations via generalized mean.} Capitalizing on the geometric approach, we realize that one can also broadly view reversiblizations as a suitable mean or average between $L$ and its $\pi$-dual $L_{\pi}$. In this approach, we generate new reversiblizations by investigating generalized notions of means such as the Cauchy mean or the dual mean reversiblizations. Unlike the geometric projection approach, the reversiblizations generated in this approach do not typically coincide with a quasi-arithmetic mean, and are usually based on the differences between $L$ and $L_{\pi}$.
	
	\item \textbf{Generating reversiblizations via balancing function and convex $f$.} The reversiblizations generated in the first two approaches all fall into the locally-balanced Markov processes framework. To adapt this framework to generate reversiblizations, it amounts to choosing a suitable balancing function, and a rich source of such balancing functions comes from a simple average between a convex $f$ and its convex $*$-conjugate $f^*$ (to be introduced in Section \ref{sec:prelim}).
\end{enumerate}


The rest of this paper is organized as follow. We begin our paper by introducing various notions and notations in Section \ref{sec:prelim}. We proceed to discuss the geometric projection approach to generate reversiblizations in Section \ref{sec:geninfog}. Within this section, we first discuss the bisection property, and we follow with an investigation of a range of commonly used $f$-divergences and the R\'{e}nyi-divergences. We state the Markov chain version of AM-GM-HM inequality in Section \ref{subsec:MCAMGMHM}, and the notion of $f$ and $f^*$-projection centroids of a sequence of Markov chains is given in Section \ref{subsec:centroid}. In Section \ref{sec:gmean}, we discuss the generalized mean approach to generate reversiblizations. We first introduce two broad classes of Cauchy mean reversiblizations such as the Stolarsky mean and the logarithmic mean reversiblizations in Section \ref{subsec:cauchy}.
In Section \ref{subsec:dual}, we then consider dual mean reversiblizations such as the dual power mean, the dual Stolarsky mean and the dual logarithmic mean. Finally, we combine the locally-balanced Markov processes framework with the convex $*$-conjugate in $f$-divergence to generate reversiblizations in Section \ref{sec:third}.

\section{Preliminaries}\label{sec:prelim}

Let $f : \mathbb{R}_+ \to \mathbb{R}_+$ be a convex function with $f(1) = 0$ that grows with at most polynomial order. Let $\mathcal{L}$ denote the set of Markov infinitesimal generators defined on a finite state space $\mathcal{X}$, that is, the set of $\mathcal{X} \times \mathcal{X}$ matrices with non-negative off-diagonal entries and zero row sums for all rows. Similarly, we write $\mathcal{L}(\pi) \subseteq \mathcal{L}$ to be the set of reversible generators with respect to a distribution $\pi$. We say that $L$ is $\pi$-stationary if $\pi L = 0$. Let $L_{\pi}$ be the $\pi$-dual of $L \in \mathcal{L}$ in the sense of \cite[Proposition $1.2$]{JK14} with $H(x,y) = \pi(y)$ for all $x,y \in \mathcal{X}$ therein
with off-diagonal entries defined to be, for $x \neq y$, 
$$L_{\pi}(x,y) = \dfrac{\pi(y)}{\pi(x)}L(y,x),$$
while the diagonal entries of $L_{\pi}$ are such that the row sums are zero for each row. In the special case when $L$ admits $\pi$ as its unique stationary distribution, then $L_{\pi} = L^*$, the $\ell^2(\pi)$ adjoint of $L$ or the time-reversal of $L$. Note that $\ell^2(\pi)$ is the usual weighted $\ell^2$ Hilbert space endowed with the inner product $\langle \cdot, \cdot \rangle_{\pi}$, see \eqref{eq:innerprod} below. Following the definition as in \cite{DM09}, given a fixed target $\pi$, for any two given Markov infinitesimal generators $M, L \in \mathcal{L}$, we define the $f$-divergence between $M$ and $L$ to be
\begin{align}\label{def:fdivML}
	D_f(M || L) :=  \sum_{x \in \mathcal{X}} \pi(x) \sum_{y \in \mathcal{X}\backslash\{x\}} L(x,y) f\left(\dfrac{M(x,y)}{L(x,y)}\right),
\end{align}
where the convention that \textcolor{black}{$0 f(a/0) = 0$ for $a \geq 0$} applies in the definition above. We remark that by requiring non-negativity of $f$, the definition of $f$-divergence between Markov generators is slightly more restrictive than the classical definition of $f$-divergence in information theory between probability measures, see e.g. \cite{SV16} and the references therein. For instance, \textcolor{black}{the mapping $t \mapsto t \ln t$} is not in the set while $f(t) = t \ln t - t + 1$ is in the set. Let $f^*$ be the convex $*$-conjugate (or simply conjugate) of $f$ defined to be $f^*(t) := tf(1/t)$ for $t > 0$, then it can readily be seen that
\begin{align*}
	D_f(M || L) = D_{f^*}(L || M),
\end{align*}
and $f^*$ is also convex with $f^*(1) = 0$. Thus, for convex $f$ that is self-conjugate, that is, $f^* = f$, the $f$-divergence as defined in \eqref{def:fdivML} is symmetric in its arguments. As a result, we can symmetrize a possibly non-symmetric $D_f$ into a symmetric one by considering $D_{(f+f^*)/2}$. For given $L, M \in \mathcal{L}$, it will also be convenient to define
\begin{align}\label{def:Dbar}
	\overline{D}_f(L||M) := D_f\left(L\bigg|\bigg|\frac{1}{2}(L+M)\right).
\end{align}
Information divergences that can be expressed by $\overline{D}_f$ include the Jensen-Shannon divergence and Vincze-Le Cam divergence, see Section \ref{subsec:JSdLCd}.

Given a general Markov generator $L$ which does not necessarily admit $\pi$ as its stationary distribution, we are interested in investigating the projection of $L$ onto the set $\mathcal{L}(\pi)$ with respect to the $f$-divergence $D_f$ as introduced earlier in \eqref{def:fdivML}. To this end, following the notions of reversible information projections introduced in \cite{WW21} for the Kullback-Leibler divergence in a discrete-time setting, we define analogously the notions of $f$-projection and $f^*$-projection with respect to $D_f$ to be
\begin{align}\label{def:emprojection}
	M^f = M^f(L,\pi) := \argmin_{M \in \mathcal{L}(\pi)} D_f(M || L), \quad M^{f^*} = M^{f^*}(L,\pi) := \argmin_{M \in \mathcal{L}(\pi)} D_f(L || M).
\end{align}
It is instructive to note that our notions of projection are with respect to a fixed target $\pi$, while in \cite{WW21} projections are onto the entire reversible set. In the context of Markov chain Monte Carlo, we are often given a target $\pi$ for instance a posterior distribution in a Bayesian model, and in this setting it is not at all restrictive to consider and investigate projections onto $\mathcal{L}(\pi)$.

In the subsequent sections, we shall specialize in various common choices of functions $f$, and investigate the corresponding projections $M^f$ and $M^{f^*}$. It turns out that in most of these cases, these two projections can be expressed as a certain power mean of $L$ and $L_{\pi}$. We shall define, for $x \neq y \in \mathcal{X}$ and $p \in \mathbb{R}\backslash\{0\}$,
\begin{align}\label{def:Pp}
	P_{p}(x,y) := \left(\dfrac{L(x,y)^{p} + L_{\pi}(x,y)^{p}}{2}\right)^{1/p},
\end{align}
and the diagonal entries of $P_{p}$ are such that the row sum is zero for all rows, that we call power mean reversiblizations. Note that this mean also appears in \cite[equation $(2.6)$]{Amari07} in the context of $\alpha$-divergence for probability measures and is referred therein as the $\alpha$-mean. We check that $P_p$ is indeed $\pi$-reversible, since
\begin{align*}
	\pi(x) P_p(x,y) &= \left(\dfrac{(\pi(x) L(x,y))^{p} + (\pi(x) L_{\pi}(x,y))^{p}}{2}\right)^{1/p} \\
	&= \left(\dfrac{(\pi(y) L_{\pi}(y,x))^{p} + (\pi(y) L(y,x))^{p}}{2}\right)^{1/p} \\
	&= \pi(y) P_p(y,x),
\end{align*}
and hence the detailed balance condition is satisfied with $P_p$. We can also understand the limiting cases as
\begin{align*}
	P_0(x,y) &= \lim_{p \to 0} P_{p}(x,y) = \sqrt{L(x,y)L_{\pi}(x,y)}, \\
	P_{\infty}(x,y) &= \lim_{p \to \infty} P_{p}(x,y) = \max\{L(x,y),L_{\pi}(x,y)\}, \\
	P_{-\infty}(x,y) &= \lim_{p \to -\infty} P_{p}(x,y) = \min\{L(x,y),L_{\pi}(x,y)\},
\end{align*}
which are, respectively, the geometric mean reversiblization as studied in \cite{DM09} and in a discrete-time setting \cite{WW21}, $M_2$-reversiblization as proposed in \cite{Choi16}, and the classical Metropolis-Hastings reversiblization. We also call the case of $p = 1/3$ to be the Lorentz mean reversiblization as it is the Lorentz mean known in the literature \citep{Lin74}.

\section{Generating new reversiblizations via geometric projections and minimization of $f$-divergence}\label{sec:geninfog}

\subsection{A bisection property for $D_f$ and $\overline{D}_f$}

First, we present a bisection property which states that the information divergence as measured by $D_f$ is the same for the pair $(L,M)$ and $(L_{\pi},M_{\pi})$, where $M,L \in \mathcal{L}$ and we recall that $L_{\pi}$ (resp. $M_{\pi}$) is the $\pi$-dual of $L$ (resp.~$M$). This general result will be useful in proving various Pythagorean identities or bisection properties in subsequent sections.
\begin{theorem}[Bisection property of $D_f$]\label{thm:bisection}
	\textcolor{black}{Let $M,L \in \mathcal{L}$. Then we have
	\begin{align*}
		D_f(L||M) &= D_f(L_{\pi}||M_{\pi}).
	\end{align*}
	In particular, if $M \in \mathcal{L}(\pi)$ and $L \in \mathcal{L}$, this yields
	\begin{align*}
		D_f(L||M) &= D_f(L_{\pi}||M), \\
		D_f(M||L) &= D_f(M||L_{\pi}).
	\end{align*}}
\end{theorem}

\begin{proof}
	\textcolor{black}{For the first equality, we calculate that
	\begin{align*}
		D_f(L||M) &= \sum_{x \neq y} \pi(x) M(x,y) f\left(\dfrac{L(x,y)}{M(x,y)}\right) = \sum_{x \neq y} \pi(y) M_{\pi}(y,x) f\left(\dfrac{L_{\pi}(y,x)}{M_{\pi}(y,x)}\right) = D_f(L_{\pi}||M_{\pi}).
	\end{align*}}
\end{proof}

 We proceed to prove an analogous bisection property for $\overline{D}_f$:
\begin{theorem}[Bisection property of $\overline{D}_f$]\label{thm:bisectionDbar}
	\textcolor{black}{Let $M,L \in \mathcal{L}$. Then we have
	\begin{align*}
		\overline{D}_f(L||M) &= \overline{D}_f(L_{\pi}||M_{\pi}).
	\end{align*}
	In particular, if $M \in \mathcal{L}(\pi)$ and $L \in \mathcal{L}$, this yields
	\begin{align*}
		\overline{D}_f(L||M) &= \overline{D}_f(L_{\pi}||M).
	\end{align*}}	
\end{theorem}

\begin{proof}
	\textcolor{black}{We check that
	\begin{align*}
		\overline{D}_f(L||M) &= \sum_{x \neq y} \pi(x) \frac{1}{2}(L(x,y)+M(x,y)) f\left(\dfrac{L(x,y)}{\frac{1}{2}(L(x,y)+M(x,y))}\right) \\
		&= \sum_{x \neq y} \pi(y) \frac{1}{2}(L_{\pi}(y,x)+M_{\pi}(y,x)) f\left(\dfrac{L_{\pi}(y,x)}{\frac{1}{2}(L_{\pi}(y,x)+M_{\pi}(y,x))}\right) = \overline{D}_f(L_{\pi}||M_{\pi}).
	\end{align*}
	Indeed the proof shows that this remains true when $(L + M)/2$ is replaced with any convex combination.}
\end{proof}

\subsection{$\alpha$-divergence}\label{subsec:alphadiv}

In this subsection, we investigate the $f$ and $f^*$-projections of Markov chains under the $\alpha$-divergence generated by $$f_{\alpha}(t) := \frac{t^{\alpha} - \alpha t - (1-\alpha)}{\alpha(\alpha-1)},$$
where $\alpha \in \mathbb{R}\backslash\{0,1\}$. Note that $\alpha$-divergences form an important family of $f$-divergences that arises naturally in the information geometry literature \citep{Amari16}. We shall write $\alpha^* := 1 - \alpha$, and we see that $f^*_{\alpha} = f_{\alpha^*}$. Denote by
\begin{align*}
	D_\alpha := D_{f_{\alpha}}, \quad \quad D_{\alpha^*} := D_{f_{\alpha^*}}
\end{align*}
to be respectively the $\alpha$-divergence and the divergence generated by the conjugate $f_{\alpha^*}$.

We shall inspect two important cases of $\alpha$-divergence by choosing some special values of $\alpha$. In the first special case, we choose $\alpha = 2$, and we see that
$$f_{2}(t) = (t-1)^2.$$
This divergence is known as the $\chi^2$-divergence in the literature, and we shall denote by 
\begin{align*}
	D_{\chi^2} := D_{f_{2}}
\end{align*}
to be the $\chi^2$-divergence.

In the second special case, we let $\alpha = \alpha^* = 1/2$, and so we have
$$f_{1/2}(t) = 2(\sqrt{t}-1)^2.$$
The divergence generated by $(1/2)f_{1/2}$ is known as the squared Hellinger distance which we denote by
\begin{align*}
	D_{H} := D_{\frac{1}{2}f_{1/2}}.
\end{align*}
Note that $D_H$ is symmetric in its arguments since $\alpha = \alpha^* = 1/2$ and hence for all $M,L \in \mathcal{L}$, we have $D_H(M||L) = D_H(L||M)$.

With the above notations in mind, we first present the main result of this subsection, where we identify $P_{\alpha}$ and $P_{\alpha^*}$, two power mean reversiblizations with index $\alpha$ and $\alpha^*$ respectively, to be the appropriate $f_{\alpha}$ or $f_{\alpha^*}$-projections and state the associated bisection property and parallelogram laws. The proof is deferred to Section \ref{subsubsec:alphaproof}.

\begin{theorem}[$\alpha$-divergence, $P_{\alpha}$-reversiblization and $P_{\alpha^*}$-reversiblization]\label{thm:alpha}
	Suppose that $\alpha \in \mathbb{R}\backslash\{0,1\}$, $\alpha^* = 1 - \alpha$ and $L \in \mathcal{L}$.
	\begin{enumerate}	
		\item($P_{\alpha^*}$-reversiblization as $f_{\alpha}$-projection of $D_{\alpha}$ and $f_{\alpha^*}$-projection of $D_{\alpha^*}$)\label{it:alphae} The mapping 
		$$\mathcal{L}(\pi) \ni M \mapsto D_{\alpha}(M||L) ~(\textrm{resp.}~D_{\alpha^*}(L||M))$$
		admits a unique minimizer the $f_{\alpha}$-projection of $D_{\alpha}$ (resp.~ $f_{\alpha^*}$-projection of $D_{\alpha^*}$) given by, for $x \neq y \in \mathcal{X}$,
		$$	M^{f}(x,y) = \left(\dfrac{L(x,y)^{\alpha^*} + L_{\pi}(x,y)^{\alpha^*}}{2}\right)^{1/\alpha^*} = P_{\alpha^*}(x,y),$$
		the power mean $P_{\alpha^*}$ of $L(x,y)$ and $L_{\pi}(x,y)$ with $p = \alpha^*$.
		In particular, when $L$ admits $\pi$ as its stationary distribution,
		$$	M^{f}(x,y) = \left(\dfrac{L(x,y)^{\alpha^*} + L^*(x,y)^{\alpha^*}}{2}\right)^{1/\alpha^*}.$$
		
		\item($P_{\alpha}$-reversiblization as $f_{\alpha^*}$-projection of $D_f$ and $f_{\alpha}$-projection of $D_{\alpha^*}$)\label{it:alpham} The mapping 
		$$\mathcal{L}(\pi) \ni M \mapsto D_{\alpha}(L||M) ~(\textrm{resp.}~D_{\alpha^*}(M||L))$$
		admits a unique minimizer the $f_{\alpha^*}$-projection of $D_{\alpha}$ (resp.~ $f_{\alpha}$-projection of $D_{\alpha^*}$) given by, for $x \neq y \in \mathcal{X}$,
		$$	M^{f^*}(x,y) = \left(\dfrac{L(x,y)^{\alpha} + L_{\pi}(x,y)^{\alpha}}{2}\right)^{1/\alpha} = P_{\alpha}(x,y),$$
		the power mean $P_{\alpha}$ of $L(x,y)$ and $L_{\pi}(x,y)$ with $p = \alpha$. In particular, when $L$ admits $\pi$ as its stationary distribution,
		$$	M^{f^*}(x,y) = \left(\dfrac{L(x,y)^{\alpha} + L^*(x,y)^{\alpha}}{2}\right)^{1/\alpha}.$$
		
		\item(Pythagorean identity)\label{it:alphapy} For any $\overline{M} \in \mathcal{L}(\pi)$, we have
		\begin{align}
			D_{\alpha}(L || \overline{M}) &= D_{\alpha}(L||M^{f^*}) + D_{\alpha}(M^{f^*}||\overline{M}), \label{eq:alphapy1} \\
			D_{\alpha}(\overline{M} || L) &= D_{\alpha}(\overline{M} || M^{f}) + D_{\alpha}(M^{f} || L). \label{eq:alphapy2}
		\end{align}
		
		\item(Bisection property)\label{it:alphabi} We have
		\begin{align*}
			D_{\alpha}(L||M^{f^*}) &= D_{\alpha}(L_{\pi}||M^{f^*}), \\
			D_{\alpha}(M^{f}||L) &= D_{\alpha}(M^{f}||L_{\pi}).
		\end{align*}
		In particular, when $L$ admits $\pi$ as its stationary distribution, then
		\begin{align*}
			D_{\alpha}(L||M^{f^*}) &= D_{\alpha}(L^*||M^{f^*}), \\
			D_{\alpha}(M^{f}||L) &= D_{\alpha}(M^{f}||L^*).
		\end{align*}
		
		\item(Parallelogram law)\label{it:alphapara} For any $\overline{M} \in \mathcal{L}(\pi)$, we have
		\begin{align*}
			D_{\alpha}(L || \overline{M}) + D_{\alpha}(L_{\pi} || \overline{M}) &= 2D_{\alpha}(L||M^{f^*}) + 2D_{\alpha}(M^{f^*}||\overline{M}),  \\
			D_{\alpha}(\overline{M} || L) + D_{\alpha}(\overline{M} || L_{\pi}) &= 2D_{\alpha}(\overline{M} || M^{f}) + 2D_{\alpha}(M^{f} || L). 
		\end{align*}
	\end{enumerate}
\end{theorem}

\begin{rk}[On the consequence of Pythagorean identity and bisection property in practice]
	Suppose that we are given the task to sample from a given target distribution $\pi$. We have two $\pi$-stationary samplers: the first one has a generator $L$ and is non-reversible with adjoint $L^*$, while the second sampler has a $\pi$-reversible generator $\overline{M} \in \mathcal{L}(\pi)$.
	
	What is the information difference from $\overline{M}$ to $L$ with respect to the $\alpha$-divergence $D_{\alpha}$? One way to answer this question is to invoke the Pythagorean identity, which decompose the information divergence into
	\begin{align*}
		\underbrace{D_{\alpha}(L || \overline{M})}_{\text{information difference from $\overline{M}$ to $L$}} &= \underbrace{D_{\alpha}(L||M^{f^*})}_{\text{information difference from $M^{f^*}$ to $L$}} + \underbrace{D_{\alpha}(M^{f^*}||\overline{M})}_{\text{information difference from $\overline{M}$ to $M^{f^*}$}}.
	\end{align*}
	Another way to interpret this is that, within the set $\mathcal{L}(\pi)$, the unique closest $\pi$-reversible generator, measured in terms of $D_{\alpha}$, is $M^{f^*}$. Thus, if we are allowed to only simulate a $\pi$-reversible generator instead of $L$, we should simulate $M^{f^*}$ to minimize the information loss with respect to $D_{\alpha}$.
	
	Is there any information difference from $\overline{M}$ to $L$ versus from $\overline{M}$ to $L^*$? According to the bisection property, there is no difference when measured by $D_{\alpha}$ since
	$$D_{\alpha}(L||\overline{M}) = D_{\alpha}(L^*||\overline{M}).$$
\end{rk}

\begin{rk}[On generalizing the Pythagorean identity and parallelogram law to more general convex functions]
	The Pythagorean identity \eqref{it:alphapy} and parallelogram law \eqref{it:alphapara} are features of the Bregman geometry induced by the $\alpha$-divergence \citep{Amari09,Adam14} .
	The Pythagorean equality is not generally true for any $f$-divergence. We will later see in Section \ref{subsec:renyi} that it does not hold for R\'{e}nyi-divergence. We also mention that in Section \ref{subsec:approximate} we provide an approximate triangle inequality for a general three-times continuously differentiable convex $f$.
\end{rk}

\begin{rk}[On the parallelogram law]
	The parallelogram law listed in item \eqref{it:alphapara} can be interpreted graphically in an analogous manner as the Euclidean setting. We refer readers to \cite[equation $(7.57)$ to $(7.59)$ and Figure $7.10$]{N21} for an interpretation and visualization.
\end{rk}

For the special cases $\alpha = 2$ and $\alpha = 1/2$, we state two corollaries of Theorem \ref{thm:alpha}, which can serve as quick reference for the reader. We first consider the $\chi^2$-divergence where $\alpha = 2$, $\alpha^* = -1$, which gives the following Corollary:

\begin{corollary}[$\chi^2$-divergence, $P_{2}$-reversiblization and harmonic reversiblization]\label{thm:chi2}
	Suppose that $L \in \mathcal{L}$.
	\begin{enumerate}
		\item(Harmonic or $P_{-1}$-reversiblization as $f_2$-projection of $D_{\chi^2}$ and $f_{-1}$-projection of $D_{f_{-1}}$)\label{it:chi2e} The mapping 
		$$\mathcal{L}(\pi) \ni M \mapsto D_{\chi^2}(M||L) ~(\textrm{resp.}~D_{-1}(L||M))$$
		admits a unique minimizer the $f_{2}$-projection of $D_{\chi^2}$ (resp.~ $f_{-1}$-projection of $D_{-1}$) given by, for $x \neq y \in \mathcal{X}$,
		$$	M^{f}(x,y) = \left(\dfrac{L(x,y)^{-1} + L_{\pi}(x,y)^{-1}}{2}\right)^{-1} = P_{-1}(x,y),$$
		the power mean $P_{-1}$ of $L(x,y)$ and $L_{\pi}(x,y)$ with $p = -1$.
		
		\item($P_{2}$-reversiblization as $f_{-1}$-projection of $D_{\chi^2}$ and $f_2$-projection of $D_{f_{-1}}$)\label{it:chi2m} The mapping 
		$$\mathcal{L}(\pi) \ni M \mapsto D_{\chi^2}(L||M) ~(\textrm{resp.}~D_{-1}(M||L))$$
		admits a unique minimizer the $f_{-1}$-projection of $D_{\chi^2}$ (resp.~ $f_{2}$-projection of $D_{-1}$) given by, for $x \neq y \in \mathcal{X}$,
		$$	M^{f^*}(x,y) = \left(\dfrac{L(x,y)^{2} + L_{\pi}(x,y)^{2}}{2}\right)^{1/2} = P_{2}(x,y),$$
		the power mean $P_{2}$ of $L(x,y)$ and $L_{\pi}(x,y)$ with $p = 2$. 
		
		\item(Pythagorean identity)\label{it:chi2py} For any $\overline{M} \in \mathcal{L}(\pi)$, we have
		\begin{align*}
			D_{\chi^2}(L || \overline{M}) &= D_{\chi^2}(L||M^{f^*}) + D_{\chi^2}(M^{f^*}||\overline{M}),  \\
			D_{\chi^2}(\overline{M} || L) &= D_{\chi^2}(\overline{M} || M^{f}) + D_{\chi^2}(M^{f} || L). 
		\end{align*}
		
		\item(Bisection property)\label{it:chi2bi} We have
		\begin{align*}
			D_{\chi^2}(L||M^{f^*}) &= D_{\chi^2}(L_{\pi}||M^{f^*}), \\
			D_{\chi^2}(M^{f}||L) &= D_{\chi^2}(M^{f}||L_{\pi}).
		\end{align*}
		
		\item(Parallelogram law)\label{it:chi2para} For any $\overline{M} \in \mathcal{L}(\pi)$, we have
		\begin{align*}
			D_{\chi^2}(L || \overline{M}) + D_{\chi^2}(L_{\pi} || \overline{M}) &= 2D_{\chi^2}(L||M^{f^*}) + 2D_{\chi^2}(M^{f^*}||\overline{M}),  \\
			D_{\chi^2}(\overline{M} || L) + D_{\chi^2}(\overline{M} || L_{\pi}) &= 2D_{\chi^2}(\overline{M} || M^{f}) + 2D_{\chi^2}(M^{f} || L). 
		\end{align*}
	\end{enumerate}
\end{corollary}

\begin{rk}
	We remark that the harmonic or $P_{-1}$-reversiblization is in fact the Barker proposal in the Markov chain Monte Carlo literature \cite{Z20,LZ22,VLZ22}. 
\end{rk}

As the second special case of Theorem \ref{thm:alpha}, we consider the squared Hellinger distance $D_H$ with $\alpha = \alpha^* = 1/2$ to arrive at the following Corollary:

\begin{corollary}[Squared Hellinger distance and $P_{1/2}$-reversiblization]\label{thm:hellinger}
	Suppose that $L \in \mathcal{L}$.
	\begin{enumerate}
		\item($P_{1/2}$-reversiblization as $f_{1/2}$-projection)\label{it:hellingere} The mapping 
		$$\mathcal{L}(\pi) \ni M \mapsto D_H(M||L)$$
		admits a unique minimizer the $f_{1/2}$-projection given by, for $x \neq y \in \mathcal{X}$,
		$$	M^f(x,y) = \left(\dfrac{\sqrt{L(x,y)} + \sqrt{L_{\pi}(x,y)}}{2}\right)^2 = P_{1/2}(x,y),$$
		the power mean $P_{1/2}$ of $L(x,y)$ and $L_{\pi}(x,y)$ with $p = 1/2$. 
		
		\item($P_{1/2}$-reversiblization as $f_{1/2}$-projection)\label{it:hellingerm} The mapping 
		$$\mathcal{L}(\pi) \ni M \mapsto D_H(L||M)$$
		admits a unique minimizer the $f_{1/2}$-projection given by, for $x \neq y \in \mathcal{X}$,
		$$	M^{f^*}(x,y) = \left(\dfrac{\sqrt{L(x,y)} + \sqrt{L_{\pi}(x,y)}}{2}\right)^2 = P_{1/2}(x,y).$$
		
		\item(Pythagorean identity)\label{it:hellingerpy} For any $\overline{M} \in \mathcal{L}(\pi)$, we have
		\begin{align*}
			D_H(L || \overline{M}) &= D_H(L||M^{f^*}) + D_H(M^{f^*}||\overline{M}),  \\
			D_H(\overline{M} || L) &= D_H(\overline{M} || M^{f}) + D_H(M^{f} || L). 
		\end{align*}
		
		\item(Bisection property)\label{it:hellingerbi} 
		\begin{align*}
			D_H(L||M^{f^*}) &= D_H(L_{\pi}||M^{f^*}), \\
			D_H(M^{f}||L) &= D_H(M^{f}||L_{\pi}).
		\end{align*}
		
		\item(Parallelogram law)\label{it:hellingerpara} For any $\overline{M} \in \mathcal{L}(\pi)$, we have
		\begin{align*}
			D_H(L || \overline{M}) + D_H(L_{\pi} || \overline{M}) &= 2D_H(L||M^{f^*}) + 2D_H(M^{f^*}||\overline{M}),  \\
			D_H(\overline{M} || L) + D_H(\overline{M} || L_{\pi}) &= 2D_H(\overline{M} || M^{f}) + 2D_H(M^{f} || L). 
		\end{align*}
	\end{enumerate}
\end{corollary}

\subsubsection{Proof of Theorem \ref{thm:alpha}}\label{subsubsec:alphaproof}
    We first observe that if item \eqref{it:alphapy} holds, then items \eqref{it:alphae} and \eqref{it:alpham} follow. To see that, by the Pythagorean identity and the fact that $D_{\alpha} \geq 0$, we have
	\begin{align}
		D_{\alpha}(\overline{M} || L) &= D_{\alpha}(\overline{M} || M^{f}) + D_{\alpha}(M^{f} || L) \\
		&\geq D_{\alpha}(M^{f} || L).
	\end{align}
	The above equality holds if and only if $D_{\alpha}(\overline{M} || M^{f}) = 0$ if and only if $M^f = \overline{M}$ which gives the uniqueness. Similarly, using the Pythagorean identity again we have
\begin{align}
	D_{\alpha}(L || \overline{M}) &= D_{\alpha}(L||M^{f^*}) + D_{\alpha}(M^{f^*}||\overline{M}) \\
								  &\geq D_{\alpha}(L||M^{f^*}).
\end{align}
	The equality holds if and only if $D_{\alpha}(M^{f^*}||\overline{M}) = 0$ if and only if $M^{f^*} = \overline{M}$ which gives the uniqueness.
	We proceed to prove item \eqref{it:alphapy}. To prove \eqref{eq:alphapy1}, we first calculate that
	\begin{align}
		D_f(L||\overline{M}) &= \sum_{x \neq y} \pi(x) \overline{M}(x,y) \left(\dfrac{\left(\frac{L(x,y)}{\overline{M}(x,y)}\right)^{\alpha} - \alpha \frac{L(x,y)}{\overline{M}(x,y)} - (1-\alpha)}{\alpha(\alpha-1)}\right) \nonumber \\
		&= \sum_{x \neq y} \pi(x) \overline{M}(x,y) \left(\dfrac{\left(\frac{M^{f^*}(x,y)}{\overline{M}(x,y)}\right)^{\alpha} - \alpha \frac{M^{f^*}(x,y)}{\overline{M}(x,y)} - (1-\alpha)}{\alpha(\alpha-1)}\right) \nonumber \\
		&\quad + \sum_{x \neq y} \pi(x) \overline{M}(x,y) \left(\dfrac{\left(\frac{L(x,y)^{\alpha} - (M^{f^*}(x,y))^{\alpha}}{\overline{M}(x,y)^{\alpha}}\right) - \alpha \frac{L(x,y) - M^{f^*}(x,y)}{\overline{M}(x,y)}}{\alpha(\alpha-1)}\right)\nonumber \\
		&= D_f(M^{f^*} || \overline{M}) + \sum_{x \neq y} \pi(x) \overline{M}(x,y) \left(\dfrac{\left(\frac{L(x,y)^{\alpha} - (M^{f^*}(x,y))^{\alpha}}{\overline{M}(x,y)^{\alpha}}\right) - \alpha \frac{L(x,y) - M^{f^*}(x,y)}{\overline{M}(x,y)}}{\alpha(\alpha-1)}\right). \label{eq:alphaDFLM}
	\end{align}
	Using the expression of $M^{f^*}$ we note that
	\begin{align}
		\sum_{x \neq y}\pi(x) \dfrac{L(x,y)^{\alpha} - M^{f^*}(x,y)^{\alpha}}{\overline{M}(x,y)^{\alpha-1}} &= \sum_{x \neq y}\pi(x) \dfrac{L(x,y)^{\alpha} - L_{\pi}(x,y)^{\alpha}}{2\overline{M}(x,y)^{\alpha-1}} \nonumber\\
		&=  \sum_{x \neq y}\pi(x) \dfrac{L(x,y)^{\alpha}}{2\overline{M}(x,y)^{\alpha-1}} - \sum_{x \neq y}\pi(y) \dfrac{L(y,x)^{\alpha}}{2\overline{M}(y,x)^{\alpha-1}} = 0. \label{eq:alphaDFLMsim}
	\end{align}
	Substituting \eqref{eq:alphaDFLMsim} into \eqref{eq:alphaDFLM} gives rise to
	\begin{align*}
		D_f(L||\overline{M}) 
		&= D_f(M^{f^*} || \overline{M}) + \sum_{x \neq y}\pi(x) \dfrac{(-L(x,y) + M^{f^*}(x,y))}{\alpha -1},
	\end{align*}
	and it suffices to prove the second term of the right hand side above equals to $D_f(L||M^{f^*})$, which is true since
	\begin{align*}
		D_f(L||M^{f^*}) &=  \sum_{x \neq y} \pi(x) M^{f^*}(x,y) \left(\dfrac{\left(\frac{L(x,y)}{M^{f^*}(x,y)}\right)^{\alpha} - \alpha \frac{L(x,y)}{M^{f^*}(x,y)} - (1-\alpha)}{\alpha(\alpha-1)}\right) \\
		&= \sum_{x \neq y} \pi(x) M^{f^*}(x,y) \left(\dfrac{\left(\frac{L(x,y)}{M^{f^*}(x,y)}\right)^{\alpha} - 1}{\alpha(\alpha-1)}\right) + \sum_{x \neq y}\pi(x) \dfrac{(-L(x,y) + M^{f^*}(x,y))}{\alpha -1} \\
		&= \sum_{x \neq y} \pi(x) \left(M^{f^*}(x,y)\right)^{1-\alpha} \dfrac{L(x,y)^{\alpha}}{\alpha(\alpha-1)} - \sum_{x \neq y} \pi(x) M^{f^*}(x,y) \dfrac{1}{\alpha(\alpha-1)} \\
		&\quad + \sum_{x \neq y}\pi(x) \dfrac{(-L(x,y) + M^{f^*}(x,y))}{\alpha -1} \\
		&= \sum_{x \neq y}\pi(x) \dfrac{(-L(x,y) + M^{f^*}(x,y))}{\alpha -1},
	\end{align*}
	which in the last equality we use the same argument as in \eqref{eq:alphaDFLMsim} and the definition of $M^{f^*}$. 
	
	We proceed to prove \eqref{eq:alphapy2}, which follows from \eqref{eq:alphapy1}. To see this, we calculate that
	\begin{align*}
		D_{\alpha}(\overline{M} || L) &= D_{\alpha^*}(L || \overline{M}) \\
		&= D_{\alpha^*}(L||M^{f}) + D_{\alpha^*}(M^{f}||\overline{M}) \\
		&= D_{\alpha}(\overline{M} || M^{f}) + D_{\alpha}(M^{f} || L),
	\end{align*}
	where the second equality follows from \eqref{eq:alphapy1} and $f_{\alpha}^* = f_{\alpha^*}$.

	For item \eqref{it:chi2bi}, it follows directly from the bisection property in Theorem \ref{thm:bisection} where we note that $M^f, M^{f^*} \in \mathcal{L}(\pi)$. Finally, for item \eqref{it:alphapara}, we utilize both the Pythagorean identity and bisection property to reach the desired result.

\subsection{Jensen-Shannon divergence and Vincze-Le	Cam divergence}\label{subsec:JSdLCd}

In this subsection and the next, our goal is to unravel relationships or inequalities between various $f$-divergences or statistical divergences. In particular, we shall illustrate this approach by looking into the Jensen-Shannon divergence and Vincze-Le Cam divergence, which are two symmetric divergences.

Recalling the expression of $\overline{D}_f$ \eqref{def:Dbar}, we proceed to define the two above-mentioned divergences.
\begin{definition}[Jensen-Shannon divergence \cite{Lin1991,SV16}]\label{def:JSdiv}
	Given $L, M \in \mathcal{L}$ and taking $f(t) = t \ln t - t + 1$ and $h(t) = t \ln t - (1+t)\ln((1+t)/2)$, the Jensen-Shannon divergence is defined to be
	$$JS(L||M) := \overline{D}_f(L||M) + \overline{D}_f(M||L) = D_h(L||M),$$
	where $D_{KL} := D_f$ is the classical Kullback-Leibler divergence between $M$ and $L$. Note that $JS(L||M) = JS(M||L)$.
\end{definition}
\begin{definition}[Vincze-Le Cam divergence \cite{V81,SV16,LC86}]\label{def:VLCdiv}
	Given $L, M \in \mathcal{L}$ and taking $f(t) = (t-1)^2$ and $h(t) = \frac{(t-1)^2}{1+t}$, the Vincze-Le Cam divergence is defined to be
	$$\Delta(L||M) := 2\overline{D}_f(L||M) = 2 \overline{D}_f(M||L) = D_h(L||M),$$
	where $D_f = D_{\chi^2}$ is the $\chi^2$-divergence between $M$ and $L$. Note that $\Delta(L||M) = \Delta(M||L)$.
\end{definition}

While both $JS$ and $\Delta$ can be regarded as a $h$-divergence for an appropriately, strictly convex $h$, we cannot express their projections $M^h = M^{h^*}$ with our previous approach or the one in \cite{DM09}. Using the convexity of $D_f(L||\cdot)$, we can obtain inequalities between these divergences:

\begin{theorem}[Bounding Jensen-Shannon by Kullback-Leibler]\label{thm:JSKL}
	Given $L, M \in \mathcal{L}$, $\overline{M} \in \mathcal{L}(\pi)$, and taking $f(t) = t \ln t - t + 1$ and $h(t) = t \ln t - (1+t)\ln((1+t)/2)$,
	we have
	\begin{align}\label{eq:JSKL1}
		JS(L||M) \leq \dfrac{1}{2} (D_f(L||M) + D_f(M||L)).
	\end{align}
	\textcolor{black}{Recall that $M^{h^*} = \argmin_{M \in \mathcal{L}(\pi)} D_h(L||M) = \argmin_{M \in \mathcal{L}(\pi)} D_h(M||L) = M^h$} is the unique $h^*$-projection or $h$-projection of $JS = D_h$, then
	\begin{align}\label{eq:JSKL2}
		JS(L||M^{h^*}) \leq \dfrac{1}{2} (D_f(L||\overline{M}) + D_f(\overline{M}||L)).
	\end{align}
	We also have the following bisection property for $JS$:
	$$JS(L||\overline{M}) = JS(L_{\pi}||\overline{M}).$$
\end{theorem}

\begin{proof}
	To prove \eqref{eq:JSKL1}, we note that by the convexity of $D_f$ and the property that $D_f(L||L) = D_f(M||M) = 0$,
	\begin{align*}
		JS(L||M) \leq \dfrac{1}{2} (D_f(L||M) + D_f(M||L)).
	\end{align*}
	As for \eqref{eq:JSKL2}, it follows from definition that
	$$JS(L||\textcolor{black}{M^h}) \leq JS(L||\overline{M}) \leq \dfrac{1}{2} (D_f(L||\overline{M}) + D_f(\overline{M}||L)).$$
	Finally, for the bisection property, we either apply the bisection property twice for $\overline{D}_f$ (Theorem \ref{thm:bisectionDbar}) or by the bisection property once for $D_h$.
\end{proof}

The analogous theorem of $\Delta$ is now stated, and its proof is omitted since it is very similar as that of Theorem \ref{thm:JSKL}:

\begin{theorem}[Bounding Vincze-Le Cam by $\chi^2$]
	Given $L, M \in \mathcal{L}$, $\overline{M} \in \mathcal{L}(\pi)$, and taking $f(t) = (t-1)^2$ and $h(t) = \frac{(t-1)^2}{1+t}$,
	we have
	\begin{align}
		\Delta(L||M) \leq D_{\chi^2}(L||M).
	\end{align}
	\textcolor{black}{Recall that $M^{h^*} = \argmin_{M \in \mathcal{L}(\pi)} D_h(L||M) = \argmin_{M \in \mathcal{L}(\pi)} D_h(M||L) = M^h$} is the unique $h^*$-projection or $h$-projection of $\Delta = D_h$, then
	\begin{align}
		\Delta(L||M^{h^*}) \leq D_{\chi^2}(L||\overline{M}).
	\end{align}
	We also have the following bisection property for $\Delta$:
	$$\Delta(L||\overline{M}) = \Delta(L_{\pi}||\overline{M}).$$
\end{theorem}

\subsection{R\'{e}nyi-divergence}\label{subsec:renyi}

The objective of this subsection is to investigate the projections of non-reversible Markov chains with respect to other notions of statistical divergence apart from $f$-divergence. Building upon relationships between various $f$-divergences or other statistical divergences, one can possibly construct and develop new inequalities governing the information divergences between these objects. In this subsection, we shall in particular examine the R\'{e}nyi-divergence which can be defined as a log-transformed version of the $\alpha$-divergence as introduced in Section \ref{subsec:alphadiv}.

Precisely, for $\alpha > 1$, we define the R\'{e}nyi-divergence between $M,L \in \mathcal{L}$ to be
\begin{align}
	R_{\alpha}(M||L) := \dfrac{1}{\alpha-1} \ln\left(1 + \alpha(\alpha-1)D_{\alpha}(M||L)\right).
\end{align}
where we recall that $D_\alpha$ is the $\alpha$-divergence as introduced in Section \ref{subsec:alphadiv}. Since $D_f \geq 0$ and $\alpha > 1$, we note that $R_{\alpha} \geq 0$. Interestingly, we shall see that $R_{\alpha}$ inherits both the minimization property and bisection property from that of $D_\alpha$ due to the \textcolor{black}{increasing} transformation between $R_{\alpha}$ and $D_\alpha$, while owing to the concavity ($\alpha > 1$) of the transformation, the equalities in the Pythagorean identity and parallelogram law become inequalities.

\begin{theorem}[R\'{e}nyi-divergence, $P_{\alpha}$-reversiblization and $P_{1-\alpha}$-reversiblization]\label{thm:renyi}
	Let $\alpha > 1$, $\alpha^* = 1-\alpha$ and $f(t) = \frac{t^{\alpha} - \alpha t - (1-\alpha)}{\alpha(\alpha-1)}$. Suppose that $L \in \mathcal{L}$.
	\begin{enumerate}
		\item($P_{\alpha^*}$-reversiblization)\label{it:renyie} The mapping 
		$$\mathcal{L}(\pi) \ni M \mapsto R_{\alpha}(M||L)$$
		admits a unique minimizer the power mean $P_{\alpha^*}$ of $L(x,y)$ and $L_{\pi}(x,y)$ with $p = \alpha^*$. given by, for $x \neq y \in \mathcal{X}$,
		$$	 M^f(x,y) = P_{\alpha^*}(x,y) = \left(\dfrac{L(x,y)^{\alpha^*} + L_{\pi}(x,y)^{\alpha^*}}{2}\right)^{1/\alpha^*},$$
		
		\item($P_{\alpha}$-reversiblization)\label{it:renyim} The mapping 
		$$\mathcal{L}(\pi) \ni M \mapsto R_{\alpha}(L||M)$$
		admits a unique minimizer the power mean $P_{\alpha}$ of $L(x,y)$ and $L_{\pi}(x,y)$ with $p = \alpha$. given by, for $x \neq y \in \mathcal{X}$,
		$$	 M^{f^*}(x,y) = P_{\alpha}(x,y) = \left(\dfrac{L(x,y)^{\alpha} + L_{\pi}(x,y)^{\alpha}}{2}\right)^{1/\alpha},$$
		
%
		\item(Pythagorean inequality)\label{it:renyipy} For any $\overline{M} \in \mathcal{L}(\pi)$, we have
		\begin{align}
			R_{\alpha}(L || \overline{M}) &\leq R_{\alpha}(L||M^{f^*}) + R_{\alpha}(M^{f^*}||\overline{M}), \label{eq:renyipy1} \\
			R_{\alpha}(\overline{M} || L) &\leq R_{\alpha}(\overline{M} || M^{f}) + R_{\alpha}(M^{f} || L), \label{eq:renyipy2}
		\end{align}
		\item(Bisection property)\label{it:renyibi} We have
		\begin{align*}
			R_{\alpha}(L||M^{f^*}) &= R_{\alpha}(L_{\pi}||M^{f^*}), \\
			R_{\alpha}(M^{f}||L) &= R_{\alpha}(M^{f}||L_{\pi}).
		\end{align*}
		
		\item(Parallelogram inequality)\label{it:renyipara} For any $\overline{M} \in \mathcal{L}(\pi)$, we have
		\begin{align*}
			R_{\alpha}(L || \overline{M}) + R_{\alpha}(L_{\pi} || \overline{M}) &\leq 2R_{\alpha}(L||M^{f^*}) + 2R_{\alpha}(M^{f^*}||\overline{M}),  \\
			R_{\alpha}(\overline{M} || L) + R_{\alpha}(\overline{M} || L_{\pi}) &\leq 2R_{\alpha}(\overline{M} || M^{f}) + 2R_{\alpha}(M^{f} || L),
		\end{align*}
	\end{enumerate}
\end{theorem}

\begin{proof}
	First, we consider the mapping, for $x \geq 0$,
	\begin{align*}
		g(x) &:= \dfrac{1}{\alpha-1} \ln(1 + \alpha(\alpha-1)x), \\
		\dfrac{d}{dx} g(x) &= \dfrac{\alpha}{1 + \alpha(\alpha-1)x}, \\
		\dfrac{d^2}{dx^2} g(x) &= - \dfrac{\alpha^2(\alpha-1)}{(1 + \alpha(\alpha-1)x)^2}.
	\end{align*}
	Thus, we see that $g$ is a strictly increasing concave function when $\alpha > 1$. Making use of Theorem \ref{thm:alpha}, we calculate that
	\begin{align*}
		R_{\alpha}(P_{1-\alpha}||L) &= g(D_f(P_{1-\alpha}||L)) \leq g(D_f(M||L)) = R_{\alpha}(M||L), \\
		R_{\alpha}(L||P_{\alpha}) &= g(D_f(L||P_{\alpha})) \leq g(D_f(L||M)) = R_{\alpha}(L||M),
	\end{align*}
	which establish the first two items. We proceed to prove item \eqref{it:renyipy}. For $\alpha > 1$, as $g$ is strictly concave with $g(0) = 0$, $g$ is thus subadditive, which together with the Pythagorean identity for $\alpha$-divergence in Theorem \ref{thm:alpha} yields
	\begin{align*}
		R_{\alpha}(L || \overline{M}) = g(D_f(L || \overline{M})) &= g(D_f(L||M^{f^*}) + D_f(M^{f^*}||\overline{M})) \\
		&\leq g(D_f(L||M^{f^*})) + g(D_f(M^{f^*}||\overline{M})) = R_{\alpha}(L||M^{f^*}) + R_{\alpha}(M^{f^*}||\overline{M}). \\
		R_{\alpha}(\overline{M}||L) = g(D_f(\overline{M}||L)) &= g(D_f(M^{f}||L) + D_f(\overline{M}||M^{f})) \\
		&\leq g(D_f(M^{f}||L)) + g(D_f(\overline{M}||M^{f})) = R_{\alpha}(M^{f}||L) + R_{\alpha}(\overline{M}||M^{f}).
	\end{align*}
	For the bisection property, it can easily be seen as $R_{\alpha}$ is a transformation by $g$ of $D_f$ and the $\alpha$-divergence enjoys the bisection property as stated in Theorem \ref{thm:alpha}. Finally, for item \eqref{it:renyipara}, we apply the previous two items, that is, both the Pythagorean inequality and the bisection property to arrive at the stated conclusion.
\end{proof}

\subsection{A Markov chain version of arithmetic-geometric-harmonic mean inequality for hitting time and mixing time parameters}\label{subsec:MCAMGMHM}

In previous subsections, we have seen that various power means $P_p$ \textcolor{black}{introduced in \eqref{def:Pp}} appear naturally as $f$ and $f^*$-projections of appropriate $f$-divergences. For example, $P_{1/2}$ appears as both the $f^*$-projection and $f$-projection under the squared Hellinger distance, while in the literature \cite{DM09,BD01,Choi16,CH18} the additive reversiblization $P_1$ and the two Metropolis-Hastings reversiblizations $P_{-\infty}$ and $P_{\infty}$ appear as projections under the total variation distance, which is a special case of the $f$-divergence by taking $f$ to be the mapping $t \mapsto |t-1|$. The aim of this subsection is to offer comparison theorems between these reversiblizations for their hitting and mixing time parameters. 

To allow for effective comparison between these reversiblizations, we recall the notion of Peskun ordering of continuous-time Markov chains. This partial ordering was first introduced by \cite{Pesk73} in the context of discrete-time Markov chains on a finite state space. Various generalizations have then been obtained, for example to general state space in \cite{Tie98}, by \cite{LM08} to continuous-time Markov chains and recently by \cite{AL21} to the non-reversible setting.

\begin{definition}[Peskun ordering]
	Suppose that we have two continuous-time Markov chains with generators $L_1, L_2 \in \mathcal{L}(\pi)$ respectively. $L_1$ is said to dominate $L_2$ off-diagonally, written as $L_1 \succeq L_2$, if for all $x \neq y \in \mathcal{X}$, we have
	$$L_1(x,y) \geq L_2(x,y).$$
\end{definition}
We write the weighted inner product with respect to $\pi$ by $\langle\cdot, \cdot \rangle_\pi$, that is, 
\begin{align}\label{eq:innerprod}
	\langle f,g\rangle_\pi=\sum_{x\in \mathcal{X}}f(x)g(x)\pi(x),
\end{align}
for any functions $f,g: \mathcal{X}\rightarrow\mathbb{R}$. We denote by $\ell^2(\pi)$ to be the weighted Hilbert space endowed with the inner product $\langle \cdot, \cdot \rangle_{\pi}$. The quadratic form of $L \in \mathcal{L}(\pi)$ can then be expressed as
\begin{align}\label{eq:quadraticform}
	\langle -L f, f \rangle_{\pi} = \dfrac{1}{2} \sum_{x,y \in \mathcal{X}} \pi(x)L(x,y)(f(x)-f(y))^2.
\end{align}
For $L \in \mathcal{L}(\pi)$, we are particularly interested in the following list of parameters that assess or quantify the speed of convergence in terms of hitting and mixing time:
\begin{itemize}
	\item(Hitting times) We write
	\begin{align*}
		\tau_{A} &= \tau_{A}(L) := \inf\{t \geq 0; X_t \in A\}
	\end{align*}
	to be the first hitting time to the set $A \subseteq \mathcal{X}$ of the chain $X = (X_t)_{t \geq 0}$ with generator $L$, and the usual convention of $\inf \emptyset = \infty$ applies. We also adapt the notation that $\tau_y := \tau_{\{y\}}$ for $y \in \mathcal{X}$. One hitting time parameter of interest is the average hitting time $t_{av}$, defined to be
	$$t_{av} = t_{av}(L,\pi) := \sum_{x,y} \mathbb{E}_x(\tau_y) \pi(x)\pi(y).$$
	The eigentime identity gives that $t_{av}$ equals to the sum of the reciprocals of the non-zero eigenvalues of $-L$, see for instance \cite{Mao04,CuiMao10}. This is also known as the random target lemma in \cite{LPW17}.
	
	\item(Spectral gap) We write the spectral gap of $L$ to be
	\begin{align}\label{def:spectralgap}
		\lambda_2 = \lambda_2(L,\pi) := \inf\big\{\langle -Lf,f \rangle_{\pi}:\ f \in \mathbb{R}^{\mathcal{X}}, \pi(f) = 0, \pi(f^2) = 1\big\}.
	\end{align}
	The relaxation time $t_{rel}$ is the reciprocal of $\lambda_2$, that is,
	$$t_{rel} = t_{rel}(L,\pi) := \dfrac{1}{\lambda_2}.$$
	We see that in the finite state space setting, $\lambda_2$ is the second smallest eigenvalue of $-L$.
	
	\item(Asymptotic variance) For a mean zero function $h$, i.e., $\pi(h)=0$, the central limit theorem for Markov processes \cite[Theorem $2.7$]{KLO12} gives $t^{-1/2} \int_0^t h(X_s)ds$ converges in probability to a Gaussian distribution with mean zero and variance
	$$\sigma^2(h,L,\pi) := -2 \langle h,g \rangle_{\pi},$$
	where $g$ solves the Poisson equation $Lg = h$.
	
\end{itemize}

With the above notions in mind, we are now ready to state the main result of this subsection:

\begin{theorem}[Peskun ordering of power mean reversiblizations and its consequences]\label{thm:Peskun}
	For $p,q \in \mathbb{R} \cup \{\pm \infty\}$ with $p < q$, for any $f \in \mathbb{R}^{\mathcal{X}}$ we have
	\begin{align*}
		P_{q} &\succeq P_p, \\
		\langle -P_q f, f \rangle_{\pi} &\geq \langle -P_p f, f \rangle_{\pi}.
	\end{align*}
	\textcolor{black}{The above equality holds if and only if $L$ is reversible with respect to $\pi$ so that $P_p = L = L^*$.} Consequently, this leads to
	\begin{enumerate}
		\item(Hitting times)\label{hit} For $\lambda > 0$ and $A \subseteq \mathcal{X}$, we have
		$$
		\mathbb{E}_{\pi}(e^{-\lambda \tau_A(P_p)}) \leq \mathbb{E}_{\pi}(e^{-\lambda \tau_A(P_q)}).
		$$
		In particular, for any $A\subseteq\mathcal{X}$,
		$$
		\mathbb{E}_{\pi}(\tau_A(P_p)) \geq \mathbb{E}_{\pi}(\tau_A(P_q)).
		$$
		Furthermore,
		$$
		t_{av}(P_p,\pi) \geq t_{av}(P_q,\pi).
		$$
	
		\item(Spectral gap) We have
		$$
		\lambda_2(P_p,\pi)\leq \lambda_2(P_q,\pi).
		$$
		That is, $$t_{rel}(P_p,\pi)\geq t_{rel}(P_q,\pi).$$

		\item(Asymptotic variance)\label{asva} For $h \in \ell^2_0(\pi) = \{h;~\pi(h)=0\}$,
		$$
		\sigma^2(h,P_p,\pi) \geq \sigma^2(h,P_q,\pi).
		$$
	\end{enumerate}
\end{theorem}

\begin{proof}
	For $q > p$, by the classical power mean inequality \citep{Lin74}, we thus have for $x \neq y \in \mathcal{X}$,
	$$P_q(x,y) \geq P_p(x,y),$$
	which consequently yields, according to \eqref{eq:quadraticform}, 
	\begin{align*}
		P_{q} &\succeq P_p, \\
		\langle -P_q f, f \rangle_{\pi} &\geq \langle -P_p f, f \rangle_{\pi}.	
	\end{align*}
	\textcolor{black}{The power mean equality holds if and only if $P_q(x,y) = P_p(x,y)$ for all $x \neq y$ if and only if $L(x,y) = L_{\pi}(x,y)$ for all $x \neq y$ if and only if $L$ is $\pi$-reversible.} The remaining inequalities are consequences of the Peskun ordering between $P_q$ and $P_p$. Precisely, using the variational principle for the Laplace transform of hitting time as presented in \cite[Theorem $3.1$]{HM18}, we arrive at 
	$$
	\mathbb{E}_{\pi}(e^{-\lambda \tau_A(P_p)}) \leq \mathbb{E}_{\pi}(e^{-\lambda \tau_A(P_q)}).
	$$
	Subtracting by 1 on both sides and dividing by $\lambda$ followed by taking $\lambda \to 0$ gives
	$$
	\mathbb{E}_{\pi}(\tau_A(P_p)) \geq \mathbb{E}_{\pi}(\tau_A(P_q)).
	$$
	Using the variational principle for eigenvalues of $\pi$-reversible generators, each eigenvalue of $-P_q$ is greater than or equal to that of $-P_p$. By means of the eigentime identity, we see that 
	$$
	t_{av}(P_p,\pi) \geq t_{av}(P_q,\pi).
	$$
	In particular, for the second smallest eigenvalue, we have
	$$
	\lambda_2(P_p,\pi)\leq \lambda_2(P_q,\pi).
	$$
	Finally, for the asymptotic variances, the ordering readily follows from \cite[Theorem $6$]{LM08}.
\end{proof}

By comparing the power mean reversiblizations $P_p$ with $p \in \{-\infty,-1,0,1,2,\infty\}$ in the above theorem, we obtain the following Markov chain version of the classical quadratic-arithmetic-geometric-harmonic inequality:
\begin{corollary}[Markov chain version of the classical quadratic-arithmetic-geometric-harmonic inequality]\label{cor:AMGMHM}
	For $p \in \mathbb{R} \cup \{\pm \infty\}$ and $L \in \mathcal{L}$, we consider the power mean reversiblizations $P_p$ with $p \in \{-\infty,-1,0,1,2,\infty\}$ to arrive at
	\begin{enumerate}
		\item(Hitting times) For $\lambda > 0$ and $A \subseteq \mathcal{X}$, we have
		$$
		\mathbb{E}_{\pi}(e^{-\lambda \tau_A(P_{-\infty})}) \leq \mathbb{E}_{\pi}(e^{-\lambda \tau_A(P_{-1})}) \leq  \mathbb{E}_{\pi}(e^{-\lambda \tau_A(P_0)}) \leq \mathbb{E}_{\pi}(e^{-\lambda \tau_A(P_1)}) \leq \mathbb{E}_{\pi}(e^{-\lambda \tau_A(P_2)}) \leq \mathbb{E}_{\pi}(e^{-\lambda \tau_A(P_\infty)}).
		$$
		In particular, for any $A\subseteq\mathcal{X}$,
		$$
		\mathbb{E}_{\pi}(\tau_A(P_{-\infty})) \geq \mathbb{E}_{\pi}(\tau_A(P_{-1})) \geq \mathbb{E}_{\pi}(\tau_A(P_0)) \geq \mathbb{E}_{\pi}(\tau_A(P_1)) \geq \mathbb{E}_{\pi}(\tau_A(P_2)) \geq \mathbb{E}_{\pi}(\tau_A(P_{\infty})).
		$$
		Furthermore,
		$$
		t_{av}(P_{-\infty},\pi) \geq t_{av}(P_{-1},\pi) \geq t_{av}(P_0,\pi) \geq t_{av}(P_1,\pi) \geq t_{av}(P_2,\pi) \geq t_{av}(P_{\infty},\pi).
		$$
		
		\item(Spectral gap) We have
		$$
		\lambda_2(P_{-\infty},\pi)\leq \lambda_2(P_{-1},\pi) \leq \lambda_2(P_0,\pi)\leq \lambda_2(P_1,\pi) \leq
		\lambda_2(P_2,\pi)\leq \lambda_2(P_{\infty},\pi).
		$$
		That is, 
		$$t_{rel}(P_{-\infty},\pi)\geq t_{rel}(P_{-1},\pi) \geq t_{rel}(P_0,\pi)\geq t_{rel}(P_1,\pi) \geq t_{rel}(P_2,\pi)\geq t_{rel}(P_{\infty},\pi).$$

		\item(Asymptotic variance) For $h \in \ell^2_0(\pi) = \{h;~\pi(h)=0\}$,
		$$
		\sigma^2(h,P_{-\infty},\pi) \geq \sigma^2(h,P_{-1},\pi) \geq 
		\sigma^2(h,P_0,\pi) \geq \sigma^2(h,P_1,\pi) \geq 
		\sigma^2(h,P_2,\pi) \geq \sigma^2(h,P_{\infty},\pi).
		$$
	\end{enumerate}
\textcolor{black}{All the above equalities hold if $L$ is $\pi$-reversible with $L = L^*$ so that all the power mean reversiblizations $P_p$ collapse to $L$.}
\end{corollary}

In view of the above Corollary, we thus see that the power mean reversiblizations $P_p$ with $p \in \mathbb{R}$ interpolates between the two Metropolis-Hastings reversiblizations $P_{-\infty}$ and $P_{\infty}$. Corollary \ref{cor:AMGMHM} is important from at least the following three perspectives: first, it is a mathematically elegant generalization of the AM-GM-HM inequality in the context of Markov chain. Second, it offers new bounds on the spectral gap of the additive reversiblization $\lambda_2(P_1)$, which can be used to further bound the rate of convergence of the original non-reversible Markov chain in the spirit of \cite{Fill91}. Third, it offers comparison theorems for important hitting time and mixing time parameters of various reversible samplers such as $P_{-\infty}$ (Metropolis-Hastings), $P_{-1}$ (Barker proposal) and $P_{\infty}$ (the second Metropolis-Hastings as in \cite{ChoiMPRF,CH18}). This can yield practical guidance on the choice of samplers.

 We also remark that in addition to the above hitting time and mixing time parameters, we should also take into account of the transition rates for comparison between different reversiblizations, since the transition rates of the same row (i.e. the sum of off-diagonal entries of the row) are in general different between $P_p$ and $P_q$ for $p \neq q$ unless $L \in \mathcal{L}(\pi)$ is $\pi$-reversible. Interested readers should also consult the discussion in \cite[discussion above Remark $2.2$]{DM09}.

\textcolor{black}{Inspired by one of the referees' suggestions, we can in fact consider a regularized or penalized entropy minimization problem, so that the resulting projection has comparable transition rates as say $P_{-\infty}$, the classical Metropolis-Hastings reversiblization. Precisely, let $\lambda \geq 0$ be a regularization hyperparameter that controls the strength of regularization. We can consider the following $\ell^2$-regularized optimization problems given by
	\begin{align*}
		M^f(L,\pi,\lambda) &:= \argmin_{M \in \mathcal{L}(\pi)} \left(D_f(M || L) + \lambda \sum_{x \in \mathcal{X}} \left(1 + M(x,x)\right)^2\right), \\
		M^{f^*}(L,\pi,\lambda) &:= \argmin_{M \in \mathcal{L}(\pi)} \left(D_f(L || M) + \lambda \sum_{x \in \mathcal{X}} \left(1 + M(x,x) \right)^2\right).
	\end{align*}
When $\lambda = 0$, the regularization effect is zero and hence we retrieve $M^f(L,\pi) = M^f(L,\pi,0)$ and $M^{f^*}(L,\pi) = M^{f^*}(L,\pi,0)$. On the other hand, we can choose $\lambda$ to be large, which forces the row transition rates of $M^f(L,\pi,\lambda)$ and $M^{f^*}(L,\pi,\lambda)$ to be close to $1$. These resulting projections can then be compared with some baseline algorithms such as $P_{-\infty}$ for an arguably fair comparison since we have taken into account of transition rates. Note that we can also consider other types of regularization such as $\ell^1$-regularization or more generally $\ell^p$-regularization. We shall not pursue this direction in this manuscript.
}

\subsection{Approximating $f$-divergence by $\chi^2$-divergence and an approximate triangle inequality}\label{subsec:approximate}

In this subsection, inspired by the technique of approximating $f$-divergence with Taylor's expansion \cite{NN14}, we investigate approximating $f$-divergence using Taylor's expansion by $\chi^2$-divergence for sufficiently smooth $f$. In practice, one may wish to compute projections such as $D_f(L||M^{f^*})$ and $D_f(M^{f}||L)$, yet in general the $f^*$-projection $M^{f^*}$ and $f$-projection $M^f$ may not admit a closed-form. Our main result below demonstrates that $D_f(L||M)$ can be approximated by $D_{\chi^2}(L||M)$ (that is, the $\chi^2$-divergence with generator $t \mapsto (t-1)^2$) modulo a prefactor error coefficient $\dfrac{f^{\prime \prime}(1)}{2}$ and an additive error term $\dfrac{1}{3!} ||f^{(3)}||_{\infty} (\overline{m}-\underline{m})^3$ in the Theorem below:

\begin{theorem}
	For strictly convex and three-times continuously differentiable $f$, for any $L,M \in \mathcal{L}$, we have
	\begin{align}\label{eq:approxf}
		\bigg|D_f(L||M) - \dfrac{f^{\prime \prime}(1)}{2} D_{\chi^2}(L||M)\bigg| \leq \dfrac{1}{3!} ||f^{(3)}||_{\infty} (\overline{m}-\underline{m})^3,
	\end{align}
	where
	\begin{align*}
		\overline{m} = \overline{m}(L,M) &:= \max_{L(x,y),M(x,y) > 0} \dfrac{L(x,y)}{M(x,y)}, \quad \underline{m} = \underline{m}(L,M) := \min_{L(x,y),M(x,y) > 0} \dfrac{L(x,y)}{M(x,y)}, \\
		||f^{(3)}||_{\infty}(L,M) &:= \sup_{x \in [\underline{m},\overline{m}]} |f^{(3)}(x)|.
	\end{align*}
	Note that the norm $||f^{(3)}||_{\infty}(L,M)$ depends on $(L,M)$ via $\overline{m}, \underline{m}$. In particular, for any $\overline{M} \in \mathcal{L}(\pi)$ we have
	\begin{align*}
	&\left|D_f(L||\overline{M}) - \left(D_f(L||P_2) + D_f(P_2||\overline{M})\right)\right| \\
	&\leq \dfrac{1}{3!} ||f^{(3)}||_{\infty}(L,\overline{M}) (\overline{m}(L,\overline{M})-\underline{m}(L,\overline{M}))^3 + \dfrac{1}{3!} ||f^{(3)}||_{\infty}(L,P_2) (\overline{m}(L,P_2)-\underline{m}(L,P_2))^3 \\
	&\quad + \dfrac{1}{3!} ||f^{(3)}||_{\infty}(P_2,\overline{M}) (\overline{m}(P_2,\overline{M})-\underline{m}(P_2,\overline{M}))^3,
	\end{align*}

	where we recall that $P_{2}$ is the $P_2$-reversiblization as stated in Theorem \ref{thm:chi2}. Similarly, we have
	\begin{align*}
		&\left|D_f(\overline{M}||L) - \left(D_f(P_{-1}||L) + D_f(\overline{M}||P_{-1})\right)\right| \\
		&\leq \dfrac{1}{3!} ||f^{(3)}||_{\infty}(\overline{M},L) (\overline{m}(\overline{M},L)-\underline{m}(\overline{M},L))^3 + \dfrac{1}{3!} ||f^{(3)}||_{\infty}(P_{-1},L) (\overline{m}(P_{-1},L)-\underline{m}(P_{-1},L))^3 \\
		&\quad + \dfrac{1}{3!} ||f^{(3)}||_{\infty}(\overline{M},P_{-1}) (\overline{m}(\overline{M},P_{-1})-\underline{m}(\overline{M},P_{-1}))^3,
	\end{align*}
	where we recall that $P_{-1}$ is the $P_{-1}$-reversiblization as stated in Theorem \ref{thm:chi2}.
\end{theorem}

We can interpret the expression $\overline{m}(L,M) - \underline{m}(L,M)$ as quantifying the difference between the two generators $L$ and $M$. In the case when $L = M$, equality is achieved in \eqref{eq:approxf} as the right hand side yields $\overline{m}(L,M) - \underline{m}(L,M) = 0$ while the left hand side gives $D_f(L||M) = D_{\chi^2}(L||M) = 0$.

\begin{proof}
	For strictly convex and three-times continuously differentiable $f$, by the integral form of Taylor's expansion and since $f(1) = f^{\prime}(1) = 0$, we see that, for $x \in (\underline{m},\overline{m})$,
	\begin{align*}
		f(x) &= f(1) + f^{\prime}(1) (x-1) + \dfrac{f^{\prime \prime}(1)}{2}(x-1)^2 + \dfrac{1}{2!} \int_{\underline{m}}^x (x-t)^2 f^{(3)}(t)\,dt \\
		&= \dfrac{f^{\prime \prime}(1)}{2}(x-1)^2 + \dfrac{1}{2!} \int_{\underline{m}}^x (x-t)^2 f^{(3)}(t)\,dt.
	\end{align*}
	As a result, we arrive at
	\begin{align*}
		\bigg|D_f(L||M) - \dfrac{f^{\prime \prime}(1)}{2} D_{\chi^2}(L||M)\bigg| \leq \dfrac{1}{3!} ||f^{(3)}||_{\infty} (\overline{m}-\underline{m})^3.
	\end{align*}
	By applying \eqref{eq:approxf} three times each we obtain the two approximate triangle inequalities.
\end{proof}

\subsection{$f$ and $f^*$-projection centroids of a sequence of Markov chains}\label{subsec:centroid}

Given a sequence of Markov generators $(L_i)_{i=1}^n$, where $L_i \in \mathcal{L}$ for each $i = 1,\ldots,n$, what is the closest $\pi$-reversible generator(s) $M \in \mathcal{L}(\pi)$ on average, where the distance is measured in terms of $f$-divergence $D_f$? Precisely, we define the notions of $f^*$-projection centroid and $f$-projection centroid to be respectively
\begin{align*}
	M_{n}^{f^*} &= M_n^{f^*}(L_1,\ldots,L_n,\pi) := \argmin_{M \in \mathcal{L}(\pi)} \sum_{i=1}^n D_f(L_i || M),\\ M_{n}^{f} &= M_n^{f}(L_1,\ldots,L_n,\pi) := \argmin_{M \in \mathcal{L}(\pi)} \sum_{i=1}^n D_f(M || L_i).
\end{align*}
Note that in the special case of $n = 1$, the above notions reduce to $M_{1}^f = M^f$ and $M_{1}^{f^*} = M^{f^*}$ respectively as introduced in \eqref{def:emprojection}. This notion is analogous to that of empirical risk minimization or loss minimization that arises in statistics and machine learning: given $n$ pairs of $(x_i,y_i)_{i=1}^n$, what is the least square regression line that minimize the total squared residuals (i.e. $\ell^2$ loss)? In the context of Markov chains, given $n$ Markov generators $(L_i)_{i=1}^n$, we are looking for a reversible $M \in \mathcal{L}(\pi)$ that minimize the total deviation or discrepancy measured by $ \sum_{i=1}^n D_f(L_i || M)$ or $\sum_{i=1}^n D_f(M || L_i)$ with respect to $D_f$. Similar notions of information centroids have also been proposed in the literature for probability measures, see for example \cite{NB11,NN09,Nielsen2020} and the references therein.

Inspired by the graphs in \cite{BD01,CH18,WW21} and to visualize the concept of centroid, we illustrate two $f$-projection centroids in a rectangle and in an eight-sided polygon in Figure \ref{fig:centroids}. Similar graphs can be drawn for $f^*$-projection centroids but with the direction of the arrows flipped.

\begin{figure}
	\centering
	\begin{subfigure}{.6\textwidth}
		\centering
		\includegraphics[width=\linewidth]{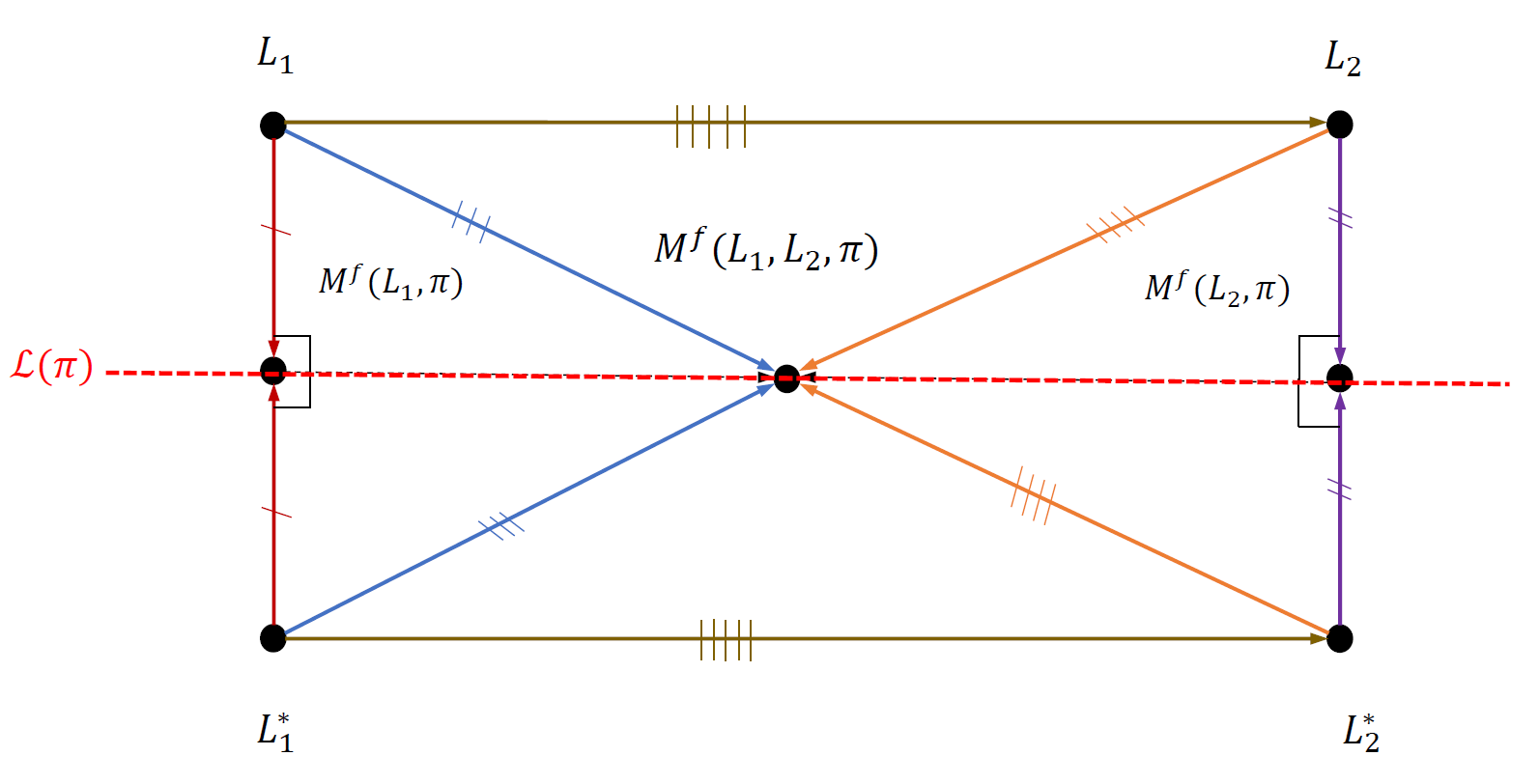}
		\caption{A rectangle generated by $L_1, L_2$ that admit $\pi$ as stationary distribution with their $\pi$-dual $L_1^*, L_2^*$. The $f$-projection centroid is $M^f_{2}(L_1,L_2,\pi)$. \textcolor{black}{Note that $D_f(L_1||L_2) = D_f(L_1^*||L_2^*)$ according to the bisection property in Theorem \ref{thm:bisection}.}}
		\label{fig:rectangle}
	\end{subfigure}%
	\vspace{0.5cm}
	\begin{subfigure}{.5\textwidth}
		\centering
		\includegraphics[width=\linewidth]{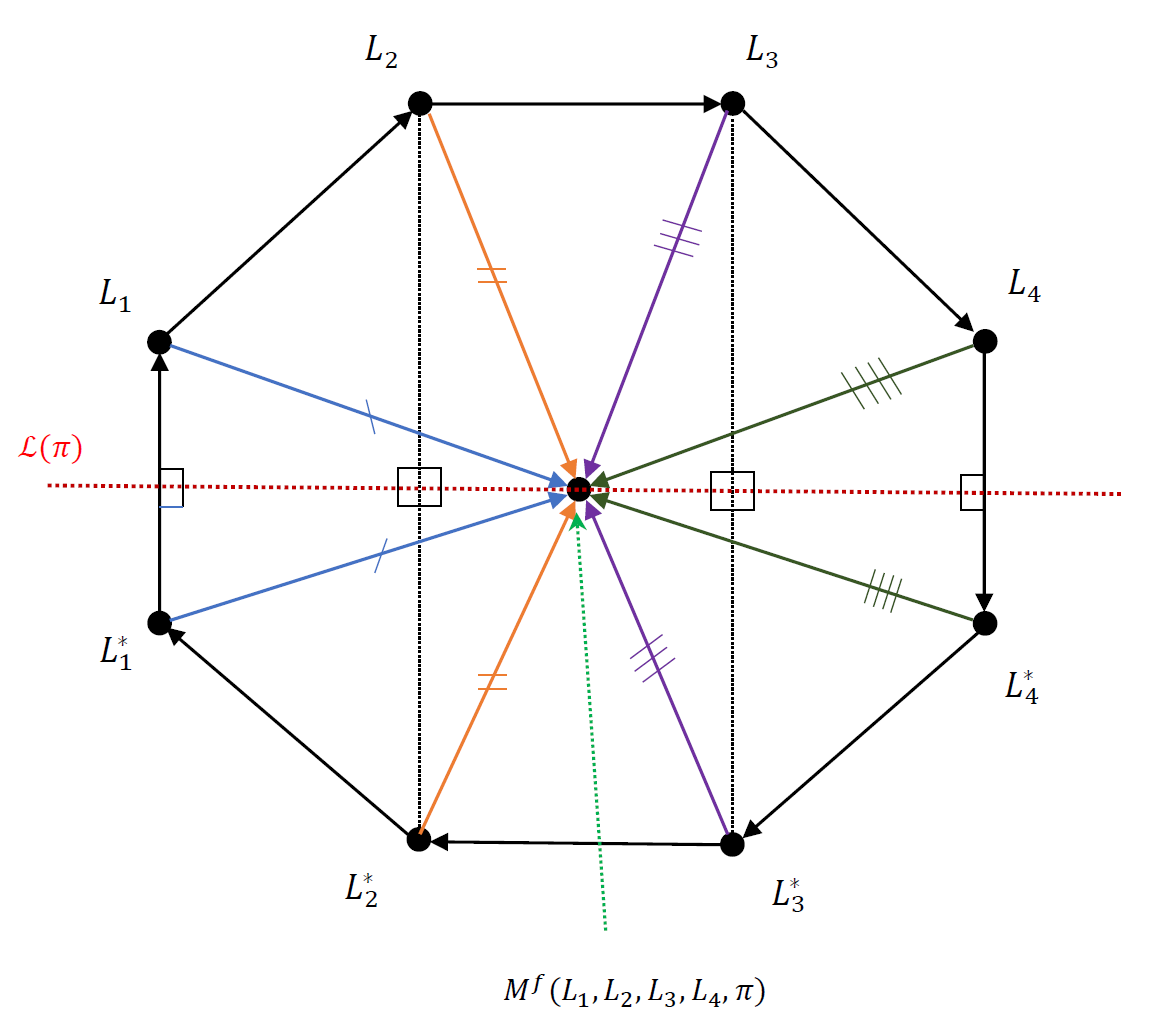}
		\caption{An eight-sided polygon generated by $L_i$ that admit $\pi$ as stationary distribution with their $\pi$-dual $L_i^*$ for $i = 1,2,3,4$. The $f$-projection centroid is $M^f(L_1,L_2,L_3,L_4,\pi)$.}
		\label{fig:polygon}
	\end{subfigure}
	\caption{Two $f$-projection centroids. The $f$-divergence under consideration $D_f$ can be any of the squared Hellinger distance, $\chi^2$-divergence, $\alpha$-divergence and Kullback-Leibler divergence as presented in Theorem \ref{thm:examplecentroid}, where both the bisection property and the Pythagorean identity have been shown. The red dashed line across the middle represents the set $\mathcal{L}(\pi)$.}
	\label{fig:centroids}
\end{figure}

Our first main result in this section proves existence and uniqueness of $f$ and $f^*$-projection centroids under strictly convex $f$, and its proof is delayed to Section \ref{subsubsec:pfexistunique}.

\begin{theorem}[Existence and uniqueness of $f$ and $f^*$-projection centroids under strictly convex $f$]\label{thm:existuniquecentroid}
	Given a sequence of Markov generators $(L_i)_{i=1}^n$, where $L_i \in \mathcal{L}$ for each $i = 1,\ldots,n$, and a $f$-divergence $D_f$ generated by a strictly convex $f$, where $f$ is assumed to have a derivative at $1$ given by $f^{\prime}(1) = 0$. A $f$-projection of $D_f$ (resp.~$f^*$-projection of $D_{f^*}$) centroid $M^f_{n}$ that minimizes the mapping 
	$$\mathcal{L}(\pi) \ni M \mapsto \sum_{i=1}^n D_f(M||L_i) \quad \left(\textrm{resp.}~=\sum_{i=1}^n D_{f^*}(L_i||M)\right)$$
	exists and is unique. A $f^*$-projection of $D_f$ (resp.~$f$-projection of $D_{f^*}$) centroid $M^{f^*}_{n}$ that minimizes the mapping 
	$$\mathcal{L}(\pi) \ni M \mapsto \sum_{i=1}^n D_f(L_i||M) \quad \left(\textrm{resp.}~=\sum_{i=1}^n D_{f^*}(M||L_i)\right)$$
	exists and is unique.
\end{theorem}

\begin{rk}
	\textcolor{black}{As we shall see in the proof of Theorem \ref{thm:existuniquecentroid}, the second part of the theorem is a consequence of the first part once it is observed that the strict convexity of $f$ is equivalent to that of $f^*$.}
\end{rk}
In the second main result of this section, we explicitly calculate the $f$ and $f^*$-projection centroids $M^f_{n}$ and $M^{f^*}_{n}$ under various common $f$-divergences as discussed in previous sections. Its proof is postponed to Section \ref{subsubsec:examplecentroid}.

\begin{theorem}[Examples of $f$ and $f^*$-projection centroids]\label{thm:examplecentroid}
	Given a sequence of Markov generators $(L_i)_{i=1}^n$, where $L_i \in \mathcal{L}$ for each $i = 1,\ldots,n$.
	\begin{enumerate}
		\item($f$ and $f^*$-projection centroids under $\alpha$-divergence)\label{it:alphacentroid}
		Let $f(t) = \frac{t^{\alpha} - \alpha t - (1-\alpha)}{\alpha(\alpha-1)}$ for $\alpha \in \mathbb{R}\backslash\{0,1\}$. The unique $f$-projection centroid $M^{f}_n$ is given by, for $x \neq y \in \mathcal{X}$,
		\begin{align*}
			M^{f}_n(x,y) =\left(\dfrac{1}{n}\sum_{i=1}^n \left(M^{f}(L_i,\pi)(x,y)\right)^{1-\alpha}\right)^{1/(1-\alpha)},
		\end{align*}
		while the unique $f^*$-projection centroid $M^{f^*}_n$ is given by, for $x \neq y \in \mathcal{X}$,
		\begin{align*}
			M^{f^*}_n(x,y) =\left(\dfrac{1}{n}\sum_{i=1}^n \left(M^{f^*}(L_i,\pi)(x,y)\right)^{\alpha}\right)^{1/\alpha},
		\end{align*}
		where we recall $M^{f}, M^{f^*}$ are respectively the $P_{1-\alpha}, P_{\alpha}$-reversiblizations as given in Theorem \ref{thm:alpha}.
		
		\item($f$ and $f^*$-projection centroids under $\chi^2$-divergence)\label{it:chi2centroid}
		Let $f(t) = (t-1)^2$.
		The unique $f$-projection centroid $M^{f}_n$ is given by, for $x \neq y \in \mathcal{X}$,
		\begin{align*}
			M^{f}_n(x,y) =\left(\dfrac{1}{n}\sum_{i=1}^n (M^{f}(L_i,\pi)(x,y))^{-1}\right)^{-1},
		\end{align*}
		while the unique $f^*$-projection centroid $M^{f^*}_n$ is given by, for $x \neq y \in \mathcal{X}$,
		\begin{align*}
			M^{f^*}_n(x,y) =\left(\dfrac{1}{n}\sum_{i=1}^n (M^{f^*}(L_i,\pi)(x,y))^2\right)^{1/2},
		\end{align*}
		where we recall $M^{f}, M^{f^*}$ are respectively the $P_{-1}, P_{2}$-reversiblizations as given in Corollary \ref{thm:chi2}.
		
		\item($f$ and $f^*$-projection centroids under squared Hellinger distance)\label{it:hellingercentroid}
			 Let $f(t) = (\sqrt{t}-1)^2$.
			 The unique $f$-projection centroid $M^{f}_n$ is given by, for $x \neq y \in \mathcal{X}$,
			 \begin{align}\label{eq:cephellinger}
			 	M^{f}_n(x,y) =\left(\dfrac{1}{n}\sum_{i=1}^n \sqrt{M^{f}(L_i,\pi)(x,y)}\right)^2,
			 \end{align}
		 	while the unique $f^*$-projection centroid $M^{f^*}_n$ is given by, for $x \neq y \in \mathcal{X}$,
		 	\begin{align*}
		 		M^{f^*}_n(x,y) =\left(\dfrac{1}{n}\sum_{i=1}^n \sqrt{M^{f^*}(L_i,\pi)(x,y)}\right)^2,
		 	\end{align*}
	 		where we recall $M^{f^*} = M^f$ are the $P_{1/2}$-reversiblizations as given in Corollary \ref{thm:hellinger}.

		\item($f$ and $f^*$-projection centroids under Kullback-Leibler divergence)
			Let $f(t) = t \ln t - t + 1$. The unique $f$-projection centroid $M^{f}_n$ is given by, for $x \neq y \in \mathcal{X}$,
			\begin{align*}
				M^{f}_n(x,y) = \left(\prod_{i=1}^n M^{f}(L_i,\pi)(x,y)\right)^{1/n},
			\end{align*}
			while the unique $f^*$-projection centroid $M^{f^*}_n$ is given by, for $x \neq y \in \mathcal{X}$,
			\begin{align*}
				M^{f^*}_n(x,y) = \dfrac{1}{n}\sum_{i=1}^n M^{f^*}(L_i,\pi)(x,y),
			\end{align*}
			where we recall $M^{f}, M^{f^*}$ are respectively the $P_{0}, P_{1}$-reversiblizations as given in \cite{DM09,WW21}, that are, the geometric mean and the additive reversiblizations.
	\end{enumerate}
\end{theorem}

\begin{rk}
	We note that item \eqref{it:chi2centroid} and \eqref{it:hellingercentroid} are special cases of item \eqref{it:alphacentroid} by taking $\alpha = 2$ and $\alpha = 1/2$ respectively. This is analogous to Corollary \ref{thm:chi2} and \ref{thm:hellinger} being special cases of Theorem \ref{thm:alpha}.
\end{rk}

\subsubsection{Proof of Theorem \ref{thm:existuniquecentroid}}\label{subsubsec:pfexistunique}

	The proof is essentially a generalization of \cite[Proposition $1.5$]{DM09}. Pick an arbitrary total ordering on $\mathcal{X}$ with strict inequality being denoted by $\prec$. For $i = 1,\ldots,n$, we also write
	\begin{align*}
		a &= a(x,y) = \pi(x)M(x,y), \quad	a^{\prime} = a^{\prime}(y,x) = \pi(y)M(y,x), \\
		\beta_i &= \beta_i(x,y) = \pi(x) L_i(x,y), \quad	\beta_i^{\prime} = \beta_i^{\prime}(y,x) = \pi(y)L_i(y,x).
	\end{align*}
	Using $M \in \mathcal{L}(\pi)$ which gives $a = a^{\prime}$, we then see that
	\begin{align*}
		\sum_{i=1}^n D_f(M || L_i) &= \sum_{i=1}^n \sum_{x \prec y} \pi(x) L_i(x,y) f\left(\dfrac{M(x,y)}{L_i(x,y)}\right) +  \pi(y) L_i(y,x) f\left(\dfrac{M(y,x)}{L_i(y,x)}\right) \\
		&= \sum_{i=1}^n \sum_{x \prec y} \beta_i f\left(\dfrac{a}{\beta_i}\right) + \beta_i^{\prime} f\left(\dfrac{a}{\beta_i^{\prime}}\right) \\
		&= \sum_{x \prec y} \sum_{i=1}^n  \beta_i f\left(\dfrac{a}{\beta_i}\right) + \beta_i^{\prime} f\left(\dfrac{a}{\beta_i^{\prime}}\right) \\
		&= \sum_{x \prec y} \sum_{\{i;~\beta_i>0 ~\textrm{or}~\beta_i^{\prime}>0\}}  \beta_i f\left(\dfrac{a}{\beta_i}\right) + \beta_i^{\prime} f\left(\dfrac{a}{\beta_i^{\prime}}\right) \\
		&=: \sum_{x \prec y} \Phi_{\beta_1,\ldots,\beta_n,\beta_1^{\prime},\ldots,\beta_n^{\prime}}(a).
	\end{align*}
	To minimize with respect to $M$, we are led to minimize the summand above $\phi := \Phi_{\beta_1,\ldots,\beta_n,\beta_1^{\prime},\ldots,\beta_n^{\prime}} : \mathbb{R}_+ \to \mathbb{R}_+$, where $(\beta_1,\ldots,\beta_n,\beta_1^{\prime},\ldots,\beta_n^{\prime}) \in \mathbb{R}_+^{2n}$ are assumed to be fixed. As $\phi$ is convex, we denote by $\phi_+^{\prime}$ to be its right derivative. It thus suffices to show the existence of $a_* > 0$ such that for all $a \in \mathbb{R}_+$,
	\begin{align}\label{eq:alpha*}
		\phi_+^{\prime}(a) = \begin{cases}
			< 0, \quad \textrm{ if } \quad a < a_*, \\
			> 0, \quad \textrm{ if } \quad a > a_*. 
		\end{cases}
	\end{align}
	Now, we compute that for all $a \in \mathbb{R}_+$,
	\begin{align*}
		\phi_+^{\prime}(a) = \sum_{\{i;~\beta_i>0 ~\textrm{and}~\beta_i^{\prime}>0\}}   f^{\prime}\left(\dfrac{a}{\beta_i}\right) + f^{\prime}\left(\dfrac{a}{\beta_i^{\prime}}\right) + \sum_{\{i;~\beta_i>0 ~\textrm{and}~\beta_i^{\prime}=0\}}   f^{\prime}\left(\dfrac{a}{\beta_i}\right)  + \sum_{\{i;~\beta_i=0 ~\textrm{and}~\beta_i^{\prime}>0\}}   f^{\prime}\left(\dfrac{a}{\beta_i^{\prime}}\right).
	\end{align*}
	As $\phi^{\prime}(1) = 0$ and $\phi$ is strictly convex, for sufficiently small $a > 0$ $\phi_+^{\prime}(a) < 0$ while for sufficiently large $a > 0$ $\phi_+^{\prime}(a) > 0$ and $\phi_+^{\prime}$ is increasing, we conclude that there exists a unique $a_* > 0$ such that \eqref{eq:alpha*} is satisfied.
	
	Replacing the analysis above by $f^*$, noting that $f^*$ is also a strictly convex function with $f^*(1) = f^{*\prime}(1) = 0$, the existence and uniqueness of $M^{f^*}_n$ is shown.

\subsubsection{Proof of Theorem \ref{thm:examplecentroid}}\label{subsubsec:examplecentroid}

	We shall only prove \eqref{eq:cephellinger} as the rest follows exactly the same computation procedure with different choices of $f$. Pick an arbitrary total ordering on $\mathcal{X}$ with strict inequality being denoted by $\prec$. For $i = 1,\ldots,n$, we also write
	\begin{align*}
		a &= a(x,y) = \pi(x)M(x,y), \quad	a^{\prime} = a^{\prime}(y,x) = \pi(y)M(y,x), \\
		\beta_i &= \beta_i(x,y) = \pi(x) L_i(x,y), \quad	\beta_i^{\prime} = \beta_i^{\prime}(y,x) = \pi(y)L_i(y,x).
	\end{align*}
	The $\pi$-reversibility of $M$ yields $a = a^{\prime}$, which leads to
	\begin{align*}
		\sum_{i=1}^n D_f(M || L_i) &= \sum_{i=1}^n \sum_{x \prec y} \pi(x) L_i(x,y) f\left(\dfrac{M(x,y)}{L_i(x,y)}\right) +  \pi(y) L_i(y,x) f\left(\dfrac{M(y,x)}{L_i(y,x)}\right) \\
		&= \sum_{i=1}^n \sum_{x \prec y} a - 2 \sqrt{a \beta_i} + \beta_i + a^{\prime} - 2 \sqrt{a^{\prime} \beta_i^{\prime}} + \beta_i^{\prime} \\
		&= \sum_{x \prec y} \sum_{i=1}^n  2a - 2 \sqrt{a \beta_i}  - 2 \sqrt{a \beta_i^{\prime}} + \beta_i + \beta_i^{\prime}.
	\end{align*}
	We proceed to minimize the summand of each term above, which leads to minimizing the following strictly convex mapping as a function of $a$
	$$a \mapsto \sum_{i=1}^n  2a - 2 \sqrt{a \beta_i}  - 2 \sqrt{a \beta_i^{\prime}}.$$
	\textcolor{black}{By differentiation and Theorem \ref{thm:hellinger}}, this yields
	\begin{align*}
		M^{f}_n(x,y) =\left(\dfrac{1}{n}\sum_{i=1}^n \sqrt{M^{f}(L_i,\pi)(x,y)}\right)^2.
	\end{align*}

\section{Generating new reversiblizations via generalized means and balancing functions}\label{sec:gmean}

For \textcolor{black}{$a,b \geq 0$} and $\phi: \mathbb{R} \to \mathbb{R}$ a continuous and strictly increasing function, the Kolmogorov-Nagumo-de Finetti mean or the quasi-arithmetic mean \cite{NN17, BC92, Carvalho16} is defined to be
$$K_{\phi}(a,b) := \phi^{-1}\left(\dfrac{\phi(a) + \phi(b)}{2}\right).$$
This is also known as the $\phi$-mean as in \cite{Amari07}. We recall from Section \ref{sec:geninfog} that various power mean reversiblizations $P_{\alpha}$ arise naturally as $f$ and $f^*$-projections under suitable choice of $f$-divergences, which are in fact special instances of the Kolmogorov-Nagumo-de Finetti mean between $L$ and $L_{\pi}$. For $\alpha \in \mathbb{R}\backslash\{0\}$, by considering $\phi(t) = t^{\alpha}$ for $t > 0$, we see that for $x \neq y \in \mathcal{X}$, 
$$P_{\alpha}(x,y) = K_{\phi}(L(x,y),L_{\pi}(x,y)).$$
Similarly, the geometric mean reversiblization $P_0$ can be retrieved by taking $\phi(t) = \ln t$ for $t > 0$. Thus, reversibling a given $L$ with a given target distribution $\pi$ can be broadly understood as
taking a suitable mean or average between $L$ and $L_{\pi}$. This important point of view is exploited in this section to generate possibly new reversiblizations via other notions of generalized mean. In particular, we shall investigate the Lagrange, Cauchy and dual mean.

As we shall see in subsequent subsections, to prove these generalized means are indeed reversible, we rely on the balancing function method introduced in the Markov chain Monte Carlo literature. As such, it is instructional to review these concepts before we proceed. To this end, given $L \in \mathcal{L}$, we define $F_g$ to be
\begin{align}\label{eq:Fg}
	F_g(x,y) := \begin{cases}
		L(x,y), \quad \textrm{if } L(x,y) = L_{\pi}(x,y), \\
		g(0) L_{\pi}(x,y), \quad \textrm{if } L(x,y) = 0 \textrm{ and } L_{\pi}(x,y) > 0, \\
		g\left(\dfrac{L_{\pi}(x,y)}{L(x,y)}\right) L(x,y), \quad \textrm{otherwise},
	\end{cases}
\end{align}
and diagonal entries of $F_g$ are such that the row sums are zero for all rows. $g: \mathbb{R}_+ \to \mathbb{R}_+$ is a function that satisfies $g(t) = tg(1/t)$ for $t > 0$, known as a balancing function introduced in the Markov chain Monte Carlo literature \cite{Z20,LZ22,VLZ22}. 

As an example to illustrate, consider $\alpha \in \mathbb{R}\backslash\{0\}$ and $P_{\alpha}$ to be the power mean reversiblization. By choosing, for $t > 0$,
$$g(t) = \left(\dfrac{t^{\alpha}+1}{2}\right)^{1/\alpha} = t g(1/t),$$
it is therefore a valid balancing function with
$$P_{\alpha} = F_g.$$ 
As a result, power mean reversiblizations can also be viewed under the balancing function framework with the above choice of $g$.

To see that $F_g$ is $\pi$-reversible, that is, $F_g \in \mathcal{L}(\pi)$, we check that the detailed balance condition is satisfied: for all $x,y \in \mathcal{X}$, we have
\begin{align*}
	\pi(x) F_g(x,y) &= \begin{cases}
		\pi(x)L(x,y), \quad \textrm{if } L(x,y) = L_{\pi}(x,y), \\
		g(0) \pi(x)L_{\pi}(x,y), \quad \textrm{if } L(x,y) = 0 \textrm{ and } L_{\pi}(x,y) > 0, \\
		g\left(\dfrac{L_{\pi}(x,y)}{L(x,y)}\right) \pi(x)L(x,y), \quad \textrm{otherwise},
	\end{cases} \\
	&= \begin{cases}
		\pi(y)L(y,x), \quad \textrm{if } L(x,y) = L_{\pi}(x,y), \\
		g(0) \pi(y) L(y,x), \quad \textrm{if } L(x,y) = 0 \textrm{ and } L_{\pi}(x,y) > 0, \\
		g\left(\dfrac{L_{\pi}(y,x)}{L(y,x)}\right) \pi(y)L(y,x), \quad \textrm{otherwise},
	\end{cases} \\
	&= \pi(y) F_g(y,x).
\end{align*}

\subsection{Generating new reversiblizations via Lagrange and Cauchy mean}\label{subsec:cauchy}
In this subsection, we investigate reversiblizations generated by Lagrange and Cauchy mean.

\begin{definition}[Lagrange and Cauchy mean \cite{BM98,Mat06}]\label{def:lcmean}
	Let $\phi_1, \phi_2$ be two differentiable and strictly increasing functions and the inverse of the ratio of their derivatives $\phi_1^{\prime}/\phi_2^{\prime}$ exists. For \textcolor{black}{$a,b \geq 0$}, the Cauchy mean is defined to be
	\begin{align*}
		\mathcal{C}_{\phi_1,\phi_2}(a,b) := \begin{cases}
			a, \quad a = b,\\
			\left(\dfrac{\phi_1^{\prime}}{\phi_2^{\prime}}\right)^{-1}\left(\dfrac{\phi_1(b) - \phi_1(a)}{\phi_2(b)-\phi_2(a)}\right), \quad a \neq b.
		\end{cases} 
	\end{align*}
	In particular, if we take $\phi_2(x) = x$, the Lagrange mean is defined to be
	\begin{align*}
		\mathcal{L}_{\phi_1}(a,b) := \begin{cases}
			a, \quad a = b,\\
			\phi_1^{\prime -1}\left(\dfrac{\phi_1(b) - \phi_1(a)}{b-a}\right), \quad a \neq b.
			\end{cases}
	\end{align*} 
\end{definition}

Capitalizing on the idea of Cauchy mean, we introduce a broad class of Cauchy mean reversiblizations where we take $\phi_1, \phi_2$ to be homogeneous functions:
\begin{theorem}\label{thm:cauchymean}
	Let $\phi_1, \phi_2: \mathbb{R}_+ \to \mathbb{R}_+$, $\phi_1(t) = t^p$ and $\phi_2(t) = t^q$ be two non-negative and homogeneous functions of degree $p, q$ respectively, where $p,q > 0$ and $p \neq q$. Given $L \in \mathcal{L}$, the Cauchy mean $C_{p,q}$ is $\pi$-reversible, that is, $C_{p,q} \in \mathcal{L}(\pi)$, where $C_{p,q}$ is defined to be
	\begin{align}\label{def:cauchym}
		C_{p,q}(x,y) := F_g
	\end{align}
	where $F_g$ is defined in \eqref{eq:Fg} with its balancing function $g$ given by 
	$$g(t) = \left(\dfrac{\phi_1^{\prime}}{\phi_2^{\prime}}\right)^{-1}\left(\dfrac{\phi_1(t) - \phi_1(1)}{\phi_2(t)-\phi_2(1)}\right) = \left(\dfrac{q(t^p-1)}{p(t^q-1)}\right)^{1/(p-q)}.$$
\end{theorem}
\begin{proof}
	It suffices to check that the given $g$ is a valid balancing function, which boils down to
	\begin{align*}
		tg(1/t) = \left(t^{p-q}\dfrac{q(t^{-p}-1)}{p(t^{-q}-1)}\right)^{1/(p-q)} 
		= \left(\dfrac{q(1-t^p)}{p(1-t^q)}\right)^{1/(p-q)} 
		= \left(\dfrac{q(t^p-1)}{p(t^q-1)}\right)^{1/(p-q)} = g(t).
	\end{align*}
\end{proof}

Interestingly, unlike the power mean reversiblizations $P_{\alpha}$, the Cauchy mean reversiblizations are based on possibly transformed differences such as $\phi_2(L(x,y)) - \phi_2(L_{\pi}(x,y))$. We shall discuss concrete examples of new reversiblizations of the form of $C_{p,q}$ that we call Stolarsky-type mean reversiblizations in Section \ref{subsubsec:stolarsky}.

Another class of Cauchy mean reversiblizations, that we call logarithmic mean reversiblizations, are generated by taking $\phi_1$ to be a homogeneous function while $\phi_2(x) = \ln x$. Some examples of new reversiblizations that fall into this class are discussed in Section \ref{subsubsec:log}.

\begin{theorem}\label{thm:logmean}
	Let $\phi_1: \mathbb{R}_+ \to \mathbb{R}_+$, $\phi_1(t) = t^p$ be a non-negative and homogeneous function of degree $p$, where $p > 0$, and $\phi_2(t) = \ln t$. Given $L \in \mathcal{L}$, the logarithmic mean $C_{p,\ln}$ is $\pi$-reversible, that is, $C_{p,\ln} \in \mathcal{L}(\pi)$, where $C_{p,\ln}$ is defined to be
	\begin{align}\label{def:logm}
		C_{p,\ln}(x,y) := F_g
	\end{align}
	where $F_g$ is defined in \eqref{eq:Fg} with its balancing function $g$ given by 
	$$g(t) = \left(\dfrac{\phi_1^{\prime}}{\phi_2^{\prime}}\right)^{-1}\left(\dfrac{\phi_1(t) - \phi_1(1)}{\phi_2(t)-\phi_2(1)}\right) = \left(\dfrac{t^p - 1}{p \ln t}\right)^{1/p}.$$
\end{theorem}

\begin{proof}
	We check that the $g$ is a valid balancing function: for $t > 0$, we have
	\begin{align*}
		t g(1/t) = \left(t^p\dfrac{t^{-p} - 1}{p (-\ln t)}\right)^{1/p}
		         = \left(\dfrac{1 - t^p}{p (-\ln t)}\right)^{1/p}
		         = \left(\dfrac{t^p - 1}{p \ln t}\right)^{1/p} 
		         = g(t).
	\end{align*}
\end{proof}
\begin{rk}\label{rk:logcauchymean}
Using the result that, for $t > 0$,
$$\ln t = \lim_{q \to 0} \dfrac{t^q - 1}{q},$$ we see that Theorem \ref{thm:logmean} can also be derived as a limiting case of Theorem \ref{thm:cauchymean}.
\end{rk}	

\subsubsection{Stolarsky mean reversiblizations}\label{subsubsec:stolarsky}

In this subsection, we investigate possibly new reversiblizations or recover known ones that belong to the Cauchy mean reversiblizations $C_{p,q}$ as introduced in \eqref{def:cauchym}.

In general, for $\phi_1(t) = t^p$ and $\phi_2(t) = t^q$ with $p,q \in \mathbb{R}_+$ and $p \neq q$, then \eqref{def:cauchym} gives, for $L(x,y), L_{\pi}(x,y) \neq 0$,
$$C_{p,q}(x,y) = \left(\dfrac{q(L^p(x,y) - L_{\pi}^p(x,y))}{p (L^q(x,y) - L_{\pi}^q(x,y))}\right)^{1/(p-q)}.$$

As a special case, we take $q = 1$, then the above expression or \eqref{def:cauchym} reads
$$C_{p,1}(x,y) = \left(\dfrac{L^p(x,y) - L_{\pi}^p(x,y)}{p (L(x,y) - L_{\pi}(x,y))}\right)^{1/(p-1)},$$
which is known as the Stolarsky mean \cite{NN17, S75} of $L(x,y), L_{\pi}(x,y)$. This is also an instance of the Lagrange mean as in Definition \ref{def:lcmean}. In particular, if we take $p = 2$, the Cauchy mean $C_{2,1}$ reduces to the simple average between $L(x,y)$ and $L_{\pi}(x,y)$. On the other hand, if $p \in \mathbb{N}$ with $p \geq 3$, the above expression can be simplified to
$$C_{p,1}(x,y) = \left(\dfrac{1}{p}\sum_{i=0}^{p-1} L(x,y)^{p-1-i}L_{\pi}(x,y)^i\right)^{1/(p-1)}.$$

\subsubsection{Logarithmic mean reversiblizations}\label{subsubsec:log}

In this subsection, we generate new reversiblizations that fall into the class of logarithmic mean reversiblizations as introduced in \eqref{def:logm}. \textcolor{black}{Note that this subsection can be considered as a consequence or corollary of the Section \ref{subsubsec:stolarsky} in view of Remark \ref{rk:logcauchymean}.}

Taking $\phi_1(t) = t^p$, with $p \in \mathbb{R}_+$, then \eqref{def:logm} now reads, for $L(x,y), L_{\pi}(x,y) \neq 0$,
$$C_{p,\ln}(x,y) = \left(\dfrac{L^p(x,y) - L_{\pi}^p(x,y)}{p (\ln L(x,y) - \ln L_{\pi}(x,y))}\right)^{1/p}.$$
In particular when $p = 1$, the above expression reduces to the classical logarithmic mean \cite{Lin74} of $L(x,y),L_{\pi}(x,y)$:
$$C_{1,\ln}(x,y) = \dfrac{L(x,y) - L_{\pi}(x,y)}{\ln L(x,y) - \ln L_{\pi}(x,y)}.$$
Note that this is also an instance of the Lagrange mean $\mathcal{L}_{\ln}(L(x,y),L_{\pi}(x,y))$, and does not belong to the class of quasi-arithmetic mean.

In the case of $p = 1$, using the arithmetic-logarithmic-geometric mean inequality \cite{Lin74}, we obtain that
\begin{align*}
	P_0(x,y) \leq C_{1,\ln}(x,y) \leq P_{1/3}(x,y) \leq P_1(x,y),
\end{align*}
where we recall that $P_0, P_{1/3}, P_1$ are respectively the geometric mean, Lorentz mean and additive reversiblizations. This yields the following Peskun ordering between these reversiblizations, and its proof is omitted as it is similar to Theorem \ref{thm:Peskun}.

\begin{theorem}[Markov chain version of arithmetic-logarithmic-geometric mean inequality]\label{thm:dualPeskun}\label{thm:AMLMGM}
	Given $L \in \mathcal{L}$, and recall the logarithmic mean reversiblization $C_{1,\ln}$ and the power mean reversiblizations as denoted by $P_p$ for $p \in \mathbb{R}$. We have
	\begin{align*}
	P_{1} &\succeq P_{1/3} \succeq C_{1,\ln} \succeq P_0.
	\end{align*}
	The above equalities hold if and only if $L$ is $\pi$-reversible. Consequently, this leads to
	\begin{enumerate}
		\item(Hitting times) For $\lambda > 0$ and $A \subseteq \mathcal{X}$, we have
		$$
		\mathbb{E}_{\pi}(e^{-\lambda \tau_A(P_0)}) \leq \mathbb{E}_{\pi}(e^{-\lambda \tau_A(C_{1,\ln})}) \leq \mathbb{E}_{\pi}(e^{-\lambda \tau_A(P_{1/3})}) \leq \mathbb{E}_{\pi}(e^{-\lambda \tau_A(P_1)}).
		$$
		In particular, for any $A\subseteq\mathcal{X}$,
		$$
		\mathbb{E}_{\pi}(\tau_A(P_0)) \geq \mathbb{E}_{\pi}(\tau_A(C_{1,\ln})) \geq \mathbb{E}_{\pi}(\tau_A(P_{1/3})) \geq \mathbb{E}_{\pi}(\tau_A(P_1)).
		$$
		Furthermore,
		$$
		t_{av}(P_0,\pi) \geq t_{av}(C_{1,\ln},\pi) \geq t_{av}(P_{1/3},\pi) \geq t_{av}(P_1,\pi).
		$$
		
		\item(Spectral gap) We have
		$$
		\lambda_2(P_0,\pi)\leq \lambda_2(C_{1,\ln},\pi) \leq \lambda_2(P_{1/3},\pi) \leq \lambda_2(P_1,\pi).
		$$
		That is, $$t_{rel}(P_0,\pi)\geq t_{rel}(C_{1,\ln},\pi) \geq t_{rel}(P_{1/3},\pi) \geq t_{rel}(P_1,\pi).$$

		\item(Asymptotic variance) For $h \in \ell^2_0(\pi) = \{h;~\pi(h)=0\}$,
		$$
		\sigma^2(h,P_0,\pi) \geq \sigma^2(h,C_{1,\ln},\pi) \geq \sigma^2(h,P_{1/3},\pi) \geq \sigma^2(h,P_1,\pi).
		$$
	\end{enumerate}
	The above equalities hold if $L \in \mathcal{L}(\pi)$ is $\pi$-reversible.
\end{theorem}

Theorem \ref{thm:AMLMGM} is important from the following three perspectives: first, it serves as a mathematically beautiful generalization of the arithmetic-logarithmic-geometric mean inequality in the realm of Markov chain. Second, it offers new bounds on the spectral gap of the additive reversiblization $\lambda_2(P_1)$, which can be used to further bound the rate of convergence of the original non-reversible Markov chain along the lines of \cite{Fill91}. Third, it offers comparison theorems on fundamental hitting time and mixing time parameters of various reversible samplers such as $P_{1}$ (additive reversiblization) and $C_{1,\ln}$ (logarithmic mean reversiblization). This can yield practical suggestions on the choice of samplers.

\subsection{Generating new reversiblizations via dual mean and generalized Barker proposal}\label{subsec:dual}

Another notion of mean that can be utilized to generate possibly new reversiblizations is the dual mean $\mathcal{M}^*$ of a given mean function $\mathcal{M}(\cdot,\cdot)$. \textcolor{black}{According to \cite[equation $8$]{NN17}, a function $\mathcal{M}: \mathbb{R}_+ \times \mathbb{R}_+ \to \mathbb{R}_+$ is said to be a mean function if it satisfies the innerness property given by, for any $a,b \in \mathbb{R}_+$,
$$\min\{a,b\} \leq \mathcal{M}(a,b) \leq \max\{a,b\}.$$}
The so-called dual mean $\mathcal{M}^*$ of the mean function $\mathcal{M}$ is defined to be
$$\mathcal{M}^*(a,b) := \begin{cases}
	\dfrac{ab}{\mathcal{M}(a,b)}, \quad \text{if } a > 0 \, \text{or } b > 0,\\
	0, \quad \text{otherwise}.
\end{cases}.$$
The mean function $\mathcal{M}$ is said to be symmetric if $\mathcal{M}(a,b) = \mathcal{M}(b,a)$, and homogeneous if $\mathcal{M}(\lambda a, \lambda b) = \lambda \mathcal{M}(a,b)$ for any $\lambda \geq 0$.

The following theorem proposes an approach that systematically generates reversiblizations via dual mean:

\begin{theorem}\label{thm:dualrev}
	Given $L \in \mathcal{L}$ and a non-negative, symmetric and homogeneous mean function $\mathcal{M}$. The dual mean $D_{\mathcal{M}}$ is $\pi$-reversible, that is, $D_{\mathcal{M}} \in \mathcal{L}(\pi)$, where $D_{\mathcal{M}}$ is defined to be
	\begin{align}\label{def:dualm}
		D_{\mathcal{M}} := F_g
	\end{align}
	where $F_g$ is defined in \eqref{eq:Fg} with its balancing function $g$ given by 
	$$g(t) = \begin{cases}t\mathcal{M}^*(1/t,1) = \dfrac{1}{\mathcal{M}(1/t,1)}, \quad \textrm{if } t > 0,\\
	0, \quad \textrm{if } t = 0.\end{cases}.$$
\end{theorem}

Note that in the special case when $\mathcal{M}$ is the simple average, we retrieve the Barker proposal or the harmonic reversiblization. Thus, $D_{\mathcal{M}}$ can be broadly interpreted as a generalization of the Barker proposal and possibly give rise to new reversible samplers.

\begin{proof}[Proof of Theorem \ref{thm:dualrev}]
	To see that $g$ is a valid balancing function, we see that, for $t > 0$, we have
	\begin{align*}
		t g(1/t) = \mathcal{M}^*(t,1) 
		         = \dfrac{t}{\mathcal{M}(t,1)} 
		         = \dfrac{1}{\mathcal{M}(1/t,1)} 
		         = g(t),
	\end{align*}
	where we utilize the symmetric and homogeneous property of $\mathcal{M}$ in the third equality.
\end{proof}

\subsubsection{Dual power mean reversiblizations}

In this subsection, we take, for $p \in \mathbb{R}\backslash\{0\}$,
$$\mathcal{M}(a,b) = \left(\dfrac{a^p+b^p}{2}\right)^{1/p},$$
the power mean of $a,b$ with index $p$, which is symmetric, homogeneous and non-negative for $a,b \geq 0$. \eqref{def:dualm} now reads
$$D_{\mathcal{M}}(x,y) = \dfrac{L(x,y)L_{\pi}(x,y)}{ \left(\dfrac{L(x,y)^p+L_{\pi}(x,y)^p}{2}\right)^{1/p}},$$
in which we retrieve the Barker proposal when we take $p = 1$.

Analogous to Theorem \ref{thm:Peskun}, we can develop a dual Peskun ordering between these dual power mean reversiblizations using the classical power mean inequality.

\subsubsection{Dual Stolarsky mean reversiblizations}

Recall that in Section \ref{subsubsec:stolarsky}, we introduce the Stolarsky mean, which gives, for $p,q \in \mathbb{R}_+\backslash\{0,1\}$ and $p \neq q$, 
$$\mathcal{M}(a,b) = \left(\dfrac{q(a^p - b^p)}{p (a^q - b^q)}\right)^{1/(p-q)},$$ 
which is symmetric, homogeneous and non-negative for $a,b \geq 0$. The dual Stolarsky mean reversiblization \eqref{def:dualm} now reads
$$D_{\mathcal{M}}(x,y) = \dfrac{L(x,y)L_{\pi}(x,y)}{ \left(\dfrac{q(L^p(x,y) - L_{\pi}^p(x,y))}{p (L^q(x,y) - L_{\pi}^q(x,y))}\right)^{1/(p-q)}}.$$

\subsubsection{Dual logarithmic mean reversiblizations}

Recall that in Section \ref{subsubsec:log}, we introduce the logarithmic mean, which gives, for $p \in \mathbb{R}_+\backslash\{0,1\}$, 
$$\mathcal{M}(a,b) = \left(\dfrac{a^p - b^p}{p (\ln a - \ln b)}\right)^{1/p},$$ 
which is symmetric, homogeneous and non-negative for $a,b \geq 0$. The dual logarithmic mean reversiblization \eqref{def:dualm} now reads
$$D_{\mathcal{M}}(x,y) = \dfrac{L(x,y)L_{\pi}(x,y)}{\left(\dfrac{L^p(x,y) - L_{\pi}^p(x,y)}{p (\ln L(x,y) - \ln L_{\pi}(x,y))}\right)^{1/p}}.$$

Analogous to Theorem \ref{thm:dualPeskun}, we can develop a dual Peskun ordering between these the dual logarithmic mean reversiblization and dual arithmetic mean reversiblization using the arithmetic-logarithmic-geometric mean inequality.

\subsection{Generating new reversiblizations via balancing functions}\label{sec:third}

In this subsection, we shall generate possibly new reversiblizations via the balancing function approach. Define
\begin{align}\label{eq:Fcalg}
	\mathcal{F}_g(x,y) := \begin{cases}
		0, \quad \textrm{if } L(x,y) = L_{\pi}(x,y), \\
		g(0) L_{\pi}(x,y), \quad \textrm{if } L(x,y) = 0 \textrm{ and } L_{\pi}(x,y) > 0, \\
		g\left(\dfrac{L_{\pi}(x,y)}{L(x,y)}\right) L(x,y), \quad \textrm{otherwise},
	\end{cases}
\end{align}
which is a locally-balanced Markov chain generated by a balancing function $g$. Comparing between $F_g$ as introduced in \eqref{eq:Fg} and $\mathcal{F}_g$, the difference lies in the value on the set $\{(x,y);~L(x,y) = L_{\pi}(x,y)\}$. A rich source of such $g$ is to consider $(f+f^*)/2$, where we recall $f$ is a non-negative convex function with $f^*$ being its conjugate as introduced in Section \ref{sec:prelim} which serves as a generator of the $f$-divergence $D_f$. In the following sections, we shall give a non-exhaustive list of new reversiblizations generated by a convex $f$ under this approach. We refer readers to \cite{SV16} and the references therein for other possible and common choices of $f$ that have been investigated in the information theory literature but are not listed in subsequent sections.

As pointed out by a reviewer, we see that this family of $\mathcal{F}_g$ generated by $g = (f+f^*)/2$ satisfies $g(1) = 0$. As $g$ generates an information divergence $D_g$ which quantifies the information difference, the family of reversiblizations generated by such $g$ can be broadly interpreted as a "difference" between $L$ and $L_{\pi}$. This is different from previous generalized mean type reversiblizations that we have covered in this paper which can be intuitively understood as generalized averages between $L$ and $L_{\pi}$. At $t = 1$ such that $g(1) = 0$, the "difference" between $L$ and $L_{\pi}$ is zero. On the other hand, other types of reversiblizations such as the power mean reversiblizations whose balancing function $t \mapsto \left(\dfrac{t^{\alpha}+1}{2}\right)^{1/\alpha}$ is $1$ at $t = 1$. As another example, the logarithmic mean reversiblization in Theorem \ref{thm:logmean} has a balancing function $t \mapsto \left(\dfrac{t - 1}{ \ln t}\right)^{1/p}$, which is again $1$ at $t = 1$.

Let us illustrate this with a concrete example. We take $f(t) = |t-1|$, from which the $f$-divergence generated is the total variation distance. We also see that $g = f = f^* = (f+f^*)/2$. Define $TV := \mathcal{F}_g$ and \eqref{eq:Fcalg} now reads, for $L(x,y) \neq L_{\pi}(x,y)$ and both are non-zero,
$$TV(x,y) = |L(x,y) - L_{\pi}(x,y)|,$$
that we call the total variation reversiblization. Using the equality that for $a,b \in \mathbb{R}$,
\begin{align*}
	\max\{a,b\} - \min\{a,b\} &= |a-b|, 
\end{align*}
we thus see that
\begin{align}\label{eq:maxminTV}
	\underbrace{P_{\infty}}_{\substack{\text{Power mean reversiblization} \\ \text{at } p = \infty}} - \underbrace{P_{-\infty}}_{\text{Metropolis-Hastings reversiblization}} &= \underbrace{TV}_{\text{"difference" between $L$ and $L_{\pi}$}}.
\end{align}
Furthermore, we can utilize  \eqref{eq:maxminTV} to develop new eigenvalue inequalities relating the eigenvalues of $P_{\infty}, P_{-\infty}$ and $TV$. By recalling that $\lambda_2(L,\pi)$ is the spectral gap of a given $L \in \mathcal{L}(\pi)$ as introduced in \eqref{def:spectralgap}, we thus have
$$\lambda_2(P_{\infty},\pi) \geq \lambda_2(P_{-\infty},\pi) + \lambda_2(TV,\pi).$$
Other eigenvalues can be related via the Weyl's inequality approach as in \cite{Choi16}. These eigenvalue bounds are important since we can utilize the eigenvalues of both $P_{\infty}$ and $P_{-\infty}$ to bound the convergence rate of the original non-reversible chain with generator $L$ using the pseudo-spectral gap approach as in \cite{Choi16}. This also highlights the three approaches to generate reversiblizations in this paper are interconnected. We shall not pursue this direction further in this manuscript.

\subsubsection{Squared Hellinger reversiblization}

In the second example, we take $f(t) = (\sqrt{t}-1)^2$, where the $f$-divergence generated is the squared Hellinger distance as introduced in Corollary \ref{thm:hellinger}. We also see that $f = f^* = (f+f^*)/2$, and \eqref{eq:Fcalg} becomes, for $L(x,y) \neq L_{\pi}(x,y)$ and both are non-zero,
$$\mathcal{F}_f(x,y) = \underbrace{(\sqrt{L(x,y)} - \sqrt{L_{\pi}(x,y)})^2}_{\text{"difference" between $L$ and $L_{\pi}$}},$$
that we call the squared Hellinger reversiblization. 

\subsubsection{Jensen-Shannon reversiblization}

In the third example, we take $f(t) = t \ln t - (1+t)\ln((1+t)/2)$, where the $f$-divergence generated is the Jensen-Shannon divergence as introduced in Definition \ref{def:JSdiv}. We also see that $f = f^* = (f+f^*)/2$, and \eqref{eq:Fcalg} becomes, for $L(x,y) \neq L_{\pi}(x,y)$ and both are non-zero,
$$\mathcal{F}_f(x,y) = \underbrace{L_{\pi}(x,y) \ln \dfrac{L_{\pi}(x,y)}{L(x,y)} - (L(x,y) + L_{\pi}(x,y)) \ln \left(\dfrac{L(x,y) + L_{\pi}(x,y)}{2 L(x,y)}\right)}_{\text{"difference" between $L$ and $L_{\pi}$}},$$
that we call the Jensen-Shannon reversiblization.

\subsubsection{Vincze-Le Cam reversiblization}

In the fourth example, we take $f(t) = \frac{(t-1)^2}{1+t}$, where the $f$-divergence generated is the Vincze-Le Cam divergence as introduced in Definition \ref{def:VLCdiv}. We also see that $f = f^* = (f+f^*)/2$, and \eqref{eq:Fcalg} becomes, for $L(x,y) \neq L_{\pi}(x,y)$ and both are non-zero,
$$\mathcal{F}_f(x,y) = \underbrace{\dfrac{(L(x,y) - L_{\pi}(x,y))^2}{L(x,y) + L_{\pi}(x,y)}}_{\text{"difference" between $L$ and $L_{\pi}$}},$$
that we call the Vincze-Le Cam reversiblization.

\subsubsection{Jeffrey reversiblization}

In the final example, we take $f(t) = (t-1) \ln t$, where the $f$-divergence generated is known as the Jeffrey's divergence. We also see that $f = f^* = (f+f^*)/2$, and \eqref{eq:Fcalg} becomes, for $L(x,y) \neq L_{\pi}(x,y)$ and both are non-zero,
$$\mathcal{F}_f(x,y) = \underbrace{(L(x,y) - L_{\pi}(x,y))(\ln L(x,y) - \ln L_{\pi}(x,y))}_{\text{"difference" between $L$ and $L_{\pi}$}},$$
that we call the Jeffrey reversiblization.

\section*{Acknowledgements}

We would like to acknowledge the careful reading and constructive comments of two reviewers that have improved the quality and the presentation of the manuscript. Michael Choi would like to thank the kind hospitality of Geoffrey Wolfer and RIKEN AIP for hosting him for a visit, in which this work was initiated. He would also like to thank Youjia Wang for his assistance in producing Figure \ref{fig:centroids}. He acknowledges the financial support from the startup grant of the National University of Singapore and the Yale-NUS College, and a Ministry of Education Tier 1 Grant under the Data for Science and Science for Data collaborative scheme with grant number 22-5715-P0001. Geoffrey Wolfer is supported by the Special Postdoctoral Researcher Program (SPDR) of RIKEN and by the Japan Society for the Promotion of Science KAKENHI under Grant 23K13024.
\bibliographystyle{abbrvnat}
\bibliography{thesis}

\begin{thebibliography}{43}
\providecommand{\natexlab}[1]{#1}
\providecommand{\url}[1]{\texttt{#1}}
\expandafter\ifx\csname urlstyle\endcsname\relax
  \providecommand{\doi}[1]{doi: #1}\else
  \providecommand{\doi}{doi: \begingroup \urlstyle{rm}\Url}\fi

\bibitem[Adam\v{c}\'{\i}k(2014)]{Adam14}
M.~Adam\v{c}\'{\i}k.
\newblock The information geometry of {B}regman divergences and some
  applications in multi-expert reasoning.
\newblock \emph{Entropy}, 16\penalty0 (12):\penalty0 6338--6381, 2014.

\bibitem[Aldous and Fill(2002)]{AF14}
D.~Aldous and J.~A. Fill.
\newblock Reversible {M}arkov {C}hains and {R}andom {W}alks on {G}raphs, 2002.
\newblock Unfinished monograph, recompiled 2014, available at
  \url{http://www.stat.berkeley.edu/~aldous/RWG/book.html}.

\bibitem[Amari(2007)]{Amari07}
S.-i. Amari.
\newblock {Integration of Stochastic Models by Minimizing alpha Divergence}.
\newblock \emph{Neural Computation}, 19\penalty0 (10):\penalty0 2780--2796, 10
  2007.

\bibitem[Amari(2009)]{Amari09}
S.-I. Amari.
\newblock {$\alpha$}-divergence is unique, belonging to both {$f$}-divergence
  and {B}regman divergence classes.
\newblock \emph{IEEE Trans. Inform. Theory}, 55\penalty0 (11):\penalty0
  4925--4931, 2009.

\bibitem[Amari(2016)]{Amari16}
S.-i. Amari.
\newblock \emph{Information geometry and its applications}, volume 194 of
  \emph{Applied Mathematical Sciences}.
\newblock Springer, [Tokyo], 2016.

\bibitem[Andrieu and Livingstone(2021)]{AL21}
C.~Andrieu and S.~Livingstone.
\newblock {Peskun Tierney ordering for Markovian Monte Carlo: Beyond the
  reversible scenario}.
\newblock \emph{The Annals of Statistics}, 49\penalty0 (4):\penalty0 1958 --
  1981, 2021.

\bibitem[Berger and Casella(1992)]{BC92}
R.~L. Berger and G.~Casella.
\newblock Deriving generalized means as least squares and maximum likelihood
  estimates.
\newblock \emph{The American Statistician}, 46\penalty0 (4):\penalty0 279--282,
  1992.

\bibitem[Berrone and Moro(1998)]{BM98}
L.~R. Berrone and J.~Moro.
\newblock Lagrangian means.
\newblock \emph{Aequationes Math.}, 55\penalty0 (3):\penalty0 217--226, 1998.

\bibitem[Billera and Diaconis(2001)]{BD01}
L.~J. Billera and P.~Diaconis.
\newblock A geometric interpretation of the {M}etropolis-{H}astings algorithm.
\newblock \emph{Statist. Sci.}, 16\penalty0 (4):\penalty0 335--339, 2001.

\bibitem[Choi(2020)]{Choi16}
M.~C. Choi.
\newblock Metropolis-{H}astings reversiblizations of non-reversible {M}arkov
  chains.
\newblock \emph{Stochastic Processes and their Applications}, 130\penalty0
  (2):\penalty0 1041 -- 1073, 2020.

\bibitem[Choi and Huang(2020)]{CH18}
M.~C. Choi and L.-J. Huang.
\newblock On hitting time, mixing time and geometric interpretations of
  {M}etropolis-{H}astings reversiblizations.
\newblock \emph{J. Theoret. Probab.}, 33\penalty0 (2):\penalty0 1144--1163,
  2020.

\bibitem[Choi(2021)]{ChoiMPRF}
M.~C.~H. Choi.
\newblock An improved variant of simulated annealing that converges under fast
  cooling.
\newblock \emph{Markov Process. Related Fields}, 27\penalty0 (1):\penalty0
  123--154, 2021.

\bibitem[Cui and Mao(2010)]{CuiMao10}
H.~Cui and Y.-H. Mao.
\newblock Eigentime identity for asymmetric finite {M}arkov chains.
\newblock \emph{Front. Math. China}, 5\penalty0 (4):\penalty0 623--634, 2010.

\bibitem[de~Carvalho(2016)]{Carvalho16}
M.~de~Carvalho.
\newblock Mean, what do you mean?
\newblock \emph{Amer. Statist.}, 70\penalty0 (3):\penalty0 270--274, 2016.

\bibitem[Diaconis and Miclo(2009)]{DM09}
P.~Diaconis and L.~Miclo.
\newblock On characterizations of {M}etropolis type algorithms in continuous
  time.
\newblock \emph{ALEA Lat. Am. J. Probab. Math. Stat.}, 6:\penalty0 199--238,
  2009.

\bibitem[Fill(1991)]{Fill91}
J.~A. Fill.
\newblock Eigenvalue bounds on convergence to stationarity for nonreversible
  {M}arkov chains, with an application to the exclusion process.
\newblock \emph{Ann. Appl. Probab.}, 1\penalty0 (1):\penalty0 62--87, 1991.

\bibitem[Huang and Mao(2018)]{HM18}
L.-J. Huang and Y.-H. Mao.
\newblock Variational principles of hitting times for non-reversible {M}arkov
  chains.
\newblock \emph{J. Math. Anal. Appl.}, 468\penalty0 (2):\penalty0 959--975,
  2018.

\bibitem[Jansen and Kurt(2014)]{JK14}
S.~Jansen and N.~Kurt.
\newblock On the notion(s) of duality for {M}arkov processes.
\newblock \emph{Probab. Surv.}, 11:\penalty0 59--120, 2014.

\bibitem[Komorowski et~al.(2012)Komorowski, Landim, and Olla]{KLO12}
T.~Komorowski, C.~Landim, and S.~Olla.
\newblock \emph{Fluctuations in {M}arkov processes}, volume 345 of
  \emph{Grundlehren der Mathematischen Wissenschaften [Fundamental Principles
  of Mathematical Sciences]}.
\newblock Springer, Heidelberg, 2012.

\bibitem[Le~Cam(1986)]{LC86}
L.~Le~Cam.
\newblock \emph{Asymptotic methods in statistical decision theory}.
\newblock Springer Series in Statistics. Springer-Verlag, New York, 1986.

\bibitem[Leisen and Mira(2008)]{LM08}
F.~Leisen and A.~Mira.
\newblock An extension of {P}eskun and {T}ierney orderings to continuous time
  {M}arkov chains.
\newblock \emph{Statist. Sinica}, 18\penalty0 (4):\penalty0 1641--1651, 2008.

\bibitem[Levin and Peres(2017)]{LPW17}
D.~A. Levin and Y.~Peres.
\newblock \emph{Markov chains and mixing times}.
\newblock American Mathematical Society, Providence, RI, 2017.

\bibitem[Lin(1991)]{Lin1991}
J.~Lin.
\newblock Divergence measures based on the {S}hannon entropy.
\newblock \emph{IEEE Trans. Inform. Theory}, 37\penalty0 (1):\penalty0
  145--151, 1991.

\bibitem[Lin(1974)]{Lin74}
T.~P. Lin.
\newblock The power mean and the logarithmic mean.
\newblock \emph{Amer. Math. Monthly}, 81:\penalty0 879--883, 1974.

\bibitem[Livingstone and Zanella(2022)]{LZ22}
S.~Livingstone and G.~Zanella.
\newblock The {B}arker proposal: combining robustness and efficiency in
  gradient-based {MCMC}.
\newblock \emph{J. R. Stat. Soc. Ser. B. Stat. Methodol.}, 84\penalty0
  (2):\penalty0 496--523, 2022.

\bibitem[Mao(2004)]{Mao04}
Y.-H. Mao.
\newblock The eigentime identity for continuous-time ergodic {M}arkov chains.
\newblock \emph{J. Appl. Probab.}, 41\penalty0 (4):\penalty0 1071--1080, 2004.

\bibitem[Matkowski(2006)]{Mat06}
J.~Matkowski.
\newblock On weighted extensions of {C}auchy's means.
\newblock \emph{J. Math. Anal. Appl.}, 319\penalty0 (1):\penalty0 215--227,
  2006.

\bibitem[Miclo(1997)]{M97}
L.~Miclo.
\newblock Remarques sur l'hypercontractivit\'{e} et l'\'{e}volution de
  l'entropie pour des cha\^{i}nes de {M}arkov finies.
\newblock In \emph{S\'{e}minaire de {P}robabilit\'{e}s, {XXXI}}, volume 1655 of
  \emph{Lecture Notes in Math.}, pages 136--167. Springer, Berlin, 1997.

\bibitem[Nielsen(2020)]{Nielsen2020}
F.~Nielsen.
\newblock On a generalization of the {J}ensen-{S}hannon divergence and the
  {J}ensen-{S}hannon centroid.
\newblock \emph{Entropy}, 22\penalty0 (2):\penalty0 Paper No. 221, 24, 2020.

\bibitem[Nielsen(2021)]{N21}
F.~Nielsen.
\newblock On geodesic triangles with right angles in a dually flat space.
\newblock In \emph{Progress in information geometry theory and applications},
  Signals Commun. Technol., pages 153--190. Springer, Cham, 2021.

\bibitem[Nielsen and Boltz(2011)]{NB11}
F.~Nielsen and S.~Boltz.
\newblock The {B}urbea-{R}ao and {B}hattacharyya centroids.
\newblock \emph{IEEE Trans. Inform. Theory}, 57\penalty0 (8):\penalty0
  5455--5466, 2011.

\bibitem[Nielsen and Nock(2009)]{NN09}
F.~Nielsen and R.~Nock.
\newblock Sided and symmetrized {B}regman centroids.
\newblock \emph{IEEE Trans. Inform. Theory}, 55\penalty0 (6):\penalty0
  2882--2904, 2009.

\bibitem[Nielsen and Nock(2014)]{NN14}
F.~Nielsen and R.~Nock.
\newblock On the chi square and higher-order chi distances for approximating
  f-divergences.
\newblock \emph{IEEE Signal Processing Letters}, 21\penalty0 (1):\penalty0
  10--13, 2014.

\bibitem[Nielsen and Nock(2017)]{NN17}
F.~Nielsen and R.~Nock.
\newblock Generalizing skew {J}ensen divergences and {B}regman divergences with
  comparative convexity.
\newblock \emph{IEEE Signal Processing Letters}, 24\penalty0 (8):\penalty0
  1123--1127, 2017.

\bibitem[Paulin(2015)]{Paulin15}
D.~Paulin.
\newblock Concentration inequalities for {M}arkov chains by {M}arton couplings
  and spectral methods.
\newblock \emph{Electron. J. Probab.}, 20:\penalty0 no. 79, 1--32, 2015.

\bibitem[Peskun(1973)]{Pesk73}
P.~H. Peskun.
\newblock Optimum {M}onte-{C}arlo sampling using {M}arkov chains.
\newblock \emph{Biometrika}, 60:\penalty0 607--612, 1973.

\bibitem[Sason and Verd\'{u}(2016)]{SV16}
I.~Sason and S.~Verd\'{u}.
\newblock {$f$}-divergence inequalities.
\newblock \emph{IEEE Trans. Inform. Theory}, 62\penalty0 (11):\penalty0
  5973--6006, 2016.

\bibitem[Stolarsky(1975)]{S75}
K.~B. Stolarsky.
\newblock Generalizations of the logarithmic mean.
\newblock \emph{Math. Mag.}, 48:\penalty0 87--92, 1975.

\bibitem[Tierney(1998)]{Tie98}
L.~Tierney.
\newblock A note on {M}etropolis-{H}astings kernels for general state spaces.
\newblock \emph{Ann. Appl. Probab.}, 8\penalty0 (1):\penalty0 1--9, 1998.

\bibitem[Vincze(1981)]{V81}
I.~Vincze.
\newblock On the concept and measure of information contained in an
  observation.
\newblock In \emph{Contributions to probability}, pages 207--214. Academic
  Press, New York-London, 1981.

\bibitem[Vogrinc et~al.(2022)Vogrinc, Livingstone, and Zanella]{VLZ22}
J.~Vogrinc, S.~Livingstone, and G.~Zanella.
\newblock {Optimal design of the Barker proposal and other locally balanced
  Metropolis-Hastings algorithms}.
\newblock \emph{Biometrika}, 10 2022.
\newblock asac056.

\bibitem[Wolfer and Watanabe(2021)]{WW21}
G.~Wolfer and S.~Watanabe.
\newblock Information geometry of reversible {M}arkov chains.
\newblock \emph{Inf. Geom.}, 4\penalty0 (2):\penalty0 393--433, 2021.

\bibitem[Zanella(2020)]{Z20}
G.~Zanella.
\newblock Informed proposals for local {MCMC} in discrete spaces.
\newblock \emph{J. Amer. Statist. Assoc.}, 115\penalty0 (530):\penalty0
  852--865, 2020.

\end{thebibliography}

\end{document}